\documentclass[12pt]{amsart}
\usepackage{amsfonts,amsmath,amssymb,amscd}
\usepackage{bm}

\newcommand{\newsection}[1]{\setcounter{equation}{0} \section{#1}}
\newtheorem{theorem}{Theorem}[section]
\newtheorem{lemma}[theorem]{Lemma}
\newtheorem{corollary}[theorem]{Corollary}
\theoremstyle{definition}
\newtheorem{definition}[theorem]{Definition}
\newtheorem{example}[theorem]{Example}
\newtheorem{proposition}[theorem]{Proposition}

\theoremstyle{remark}
\newtheorem{remark}[theorem]{Remark}
\newtheorem{question}[theorem]{Question}

\newcommand{\vp}{\varphi}
\def \C{\mathbb{C}}
\def \T{\mathbb{T}}
\def \D{\mathbb{D}}
\def \N{\mathbb{N}}
\def \Z{\mathbb{Z}}
\def \R{\mathbb{R}}

\newcommand{\bk}{\bm{k}}
\newcommand{\bl}{\bm{l}}
\newcommand{\z}{\bm{z}}

\newcommand{\clc}{\mathcal{C}}
\newcommand{\cll}{\mathcal{L}}
\newcommand{\clb}{\mathcal{B}}

\newcommand{\clh}{\mathcal{H}}
\newcommand{\clw}{\mathcal{W}}

\newcommand{\raro}{\rightarrow}
\numberwithin{equation}{section}

\newcommand{\inp}[2]{\langle{#1},\,{#2} \rangle}
\usepackage{xcolor}
\newcommand{\beqn}{\begin{eqnarray*}}
\newcommand{\eeqn}{\end{eqnarray*}}
\newcommand{\beq}{\begin{eqnarray}}
\newcommand{\eeq}{\end{eqnarray}}

\begin{document}

\title[Complex symmetric Toeplitz operators]{Characterizations of complex symmetric Toeplitz operators}

\author[Bhuia]{Sudip Ranjan Bhuia }
\address{Indian Statistical Institute, Statistics and Mathematics Unit, 8th Mile, Mysore Road, Bangalore, 560059, India}
\email{sudipranjanb@gmail.com}

\author[Pradhan]{Deepak Pradhan}
\address{Indian Statistical Institute, Statistics and Mathematics Unit, 8th Mile, Mysore Road, Bangalore, 560059,
India}
\email{deepak12pradhan@gmail.com}

\author[Sarkar]{Jaydeb Sarkar}
\address{Indian Statistical Institute, Statistics and Mathematics Unit, 8th Mile, Mysore Road, Bangalore, 560 059, India}
\email{jaydeb@gmail.com, jay@isibang.ac.in}



\begin{abstract}
We present complete characterizations of Toeplitz operators that are complex symmetric. This follows as a by-product of characterizations of conjugations on Hilbert spaces. Notably, we prove that every conjugation admits a canonical factorization. As a consequence, we prove that a Toeplitz operator is complex symmetric if and only if the Toeplitz operator is $S$-Toeplitz for some unilateral shift $S$ and the transpose of the Toeplitz operator matrix is equal to the matrix of the Toeplitz operator corresponding to the basis of the unilateral shift $S$. Also, we characterize complex symmetric Toeplitz operators on the Hardy space over the open unit polydisc. Our results answer the well known open question about characterizations of complex symmetric Toeplitz operators.
\end{abstract}

\maketitle

\tableofcontents

\section{Introduction}

All Hilbert spaces in this paper are complex and separable, with scalar product $\langle \cdot, \cdot \rangle$ linear in the first entry. Let $\clh$ be a Hilbert space. A map $C: \clh \rightarrow \clh$ is said to be a \textit{conjugation} if
\begin{enumerate}
\item $C$ is anti-linear, that is, $C(\alpha f+ g)= \bar{\alpha} C f + C g$ for all $\alpha \in \mathbb{C}$ and $f, g \in \clh$,
\item $C$ is involutive, that is, $C^2=I_{\clh}$, and
\item $C$ is isometric, that is, $\|Cf\|=\|f\|$ for all $f\in \clh$.
\end{enumerate}

We will denote by $\clc(\clh)$ the space of all conjugations on $\clh$. Besides general interest, added motivation for conjugations on Hilbert spaces comes from mathematical physics (cf. \cite{Bender, Moiseyev, Nesemann} and the survey \cite{garcia_prodan_putinar}). Moreover, conjugations patched up with bounded linear operators give the central object of this paper:

\begin{definition}
Let $\clh$ be a Hilbert space, $C \in \clc(\clh)$, and let $T \in \clb(\clh)$. We say that $T$ is $C$-symmetric if
\[
CT^*C = T.
\]
We say that $T$ is a complex symmetric if $T$ is $C$-symmetric for some $C \in \clc(\clh)$.
\end{definition}

Throughout this paper, $\clb(\clh)$ denotes the algebra of all bounded linear operators on $\clh$. It is well known that $T \in \clb(\clh)$ is symmetric if and only if there exists an orthonormal basis $\{f_n\}_{n \in \Lambda}$ of $\clh$ such that
\[
\langle T f_i, f_j\rangle = \langle T f_j, f_i \rangle,
\]
for all $i, j \in \Lambda$. Equivalently, this means that
\begin{equation}\label{eqn: classi: old symm}
[T]_{\{f_n\}_{n \in \Lambda}} = [T]_{\{f_n\}_{n \in \Lambda}}^t,
\end{equation}
where 
\[
[T]_{\{f_n\}_{n \in \Lambda}} = (\langle T f_j, f_i \rangle)_{i,j \in \Lambda},
\]
the formal matrix representation of $T$ with respect to the basis $\{f_n\}_{n \in \Lambda}$, and $[T]_{\{f_n\}_{n \in \Lambda}}^t$ denotes the transpose of the matrix $[T]_{\{f_n\}_{n \in \Lambda}}$. Note that $\Lambda$ is either a finite set or a countably infinite set (depending, of course, on the dimension of $\clh$).

The notion of complex symmetric operators is classic in linear analysis and mathematical physics. Nevertheless, a systematic study in this direction began only in 2006 with the work of Garcia and Putinar \cite{garcia1,garcia2}. Since then, researchers have rigorously studied questions about complex symmetric operators, and more specifically, models and concrete examples of complex symmetric operators (cf. \cite{ Garcia, GaWo, garciaCSPI, Hai, Jung} and the references therein). For instance, normal operators, binormal operators, Volterra operators, and Hankel operators are complex symmetric. And, notably, every $N \times N$ Toeplitz matrix, $N \geq 2$, is symmetric corresponding to the Toeplitz conjugation $C_{\text Toep}$ on $\mathbb{C}^N$ (see also Corollary \ref{cor: matrix symm}), where
\begin{equation}\label{eqn: can con FD}
C_{\text Toep} (z_1, z_2, \ldots, z_N) = (\overline{z}_{N}, \overline{z}_{N-1}, \ldots, \overline{z}_1).
\end{equation}
In fact, the starting example of complex symmetric operators in the seminal paper \cite[page 1286]{garcia1} is Toeplitz matrices with complex entries, an amplification of the classic work by Schur and Takagi \cite{Takagi}.

Besides, Toeplitz operators are one of the most important and most studied classical operators in mathematics including mathematical physics. The origin of Toeplitz operators (more specifically, Toeplitz matrices) can be traced back to the work of Otto Toeplitz at the beginning of the 20th century. However, the theory of Toeplitz operators has been profoundly influenced by the work of Brown and Halmos \cite{BH}, followed by a series of papers by L. Coburn, R. Douglas, I. Gohberg, D. Sarason, H. Widom, and many other mathematicians (see the monograph \cite{RGD}).

The purpose of this paper is to connect these two classes of operators. More specifically, here we aim to solve the natural question (also known to be an open question) of the characterizations of symmetricity of Toeplitz operators. Toeplitz operators are defined on the Hardy space $H^2(\D)$ \cite{rosenthal}. Recall that $H^2(\D)$ is the Hilbert space of all analytic functions $f = \sum_{n=0}^{\infty} a_n z^n$ on $\D = \{z \in \C: |z| < 1\}$ such that
\[
\|f\| := \Big(\sum_{n=0}^{\infty} |a_n|^2\Big)^{\frac{1}{2}} < \infty.
\]
We denote by $L^2(\mathbb{T})$ the Hilbert space of square integrable functions with respect to the normalized Lebesgue measure on the
unit circle $\mathbb{T}$. From the radial limits point of view (cf. Fatou’s theorem \cite{rosenthal}), one can identify $H^2(\D)$ with a closed subspace (denoted by $H^2(\D)$ again) of $L^2(\T)$ formed by all functions with vanishing negative Fourier coefficients. Also denote by $L^\infty(\T)$ the $C^*$-algebra of $\C$-valued essentially bounded Lebesgue measurable functions on $\T$.

\begin{definition}
The Toeplitz operator $T_\vp$ with symbol $\vp \in L^\infty(\T)$ is defined by $T_{\vp} f = P_{H^2(\D)} (\vp f)$ for all $f \in H^2(\D)$, or equivalently
\[
T_{\vp} = P_{H^2(\D)} L_{\vp}|_{H^2(\D)},
\]
where $L_\vp$ is the Laurent operator on $L^2(\T)$, and $P_{H^2(\D)}$ denotes the orthogonal projection of $L^2(\T)$ onto $H^2(\D)$.
\end{definition}

We are now in a position to state the central question of this paper more precisely which was also formally raised in \cite[Problem 4.5]{Guo} in the context of examples and the complexity of Toeplitz operators in the theory of symmetric operators:

\begin{question}\label{Q1}
Classify $\vp \in L^\infty(\T)$ such that $T_\vp$ is a complex symmetric operator.
\end{question}

In this paper, we give a solution to the above question. Note that up until now, this problem has been solved only by fixing a specific class of Toeplitz operators along with a specific class of conjugations (cf. \cite{Bu, Guo, Ko_Lee_SCTO}). Here we consider the above question in its full generality, that is, we deal with an arbitrary conjugation and an arbitrary Toeplitz operator at a time. Indeed, a closer look at Question \ref{Q1} reveals that the problem has two parts. First, and perhaps the most intricate one, is the precise representations of conjugations. This part is indeed relevant as a Toeplitz operator could be symmetric with respect to one conjugation but need not be with respect to another (see the examples following Theorem \ref{finiteToeplitzCS}). And even more, there are Toeplitz operators that are not symmetric with respect to any conjugations. The second question is the classification of symmetric Toeplitz operators in terms of a concrete conjugation.

We employ a couple of different approaches: First, we connect symmetric Toeplitz operators with the classical notion of $S$-Toeplitz operators \cite{RR-Book}. Let $S \in \clb(H^2(\D))$ be a unilateral shift. That is, there exists an orthonormal basis $\{f_n\}_{n \in \Z_+}$ of $H^2(\D)$ such (see Definition \ref{def: shift}) that
\[
S f_n = f_{n+1} \qquad (n \in \Z_+).
\]
When we wish to emphasize the orthonormal basis, we often call $S$  the \textit{shift corresponding to the basis $\{f_n\}_{n \in \Z_+}$}. The simplest example to give is the multiplication operator $M_z$ on $H^2(\D)$, where
\[
M_z f = z f \qquad (f \in H^2(\D)).
\]
An operator $T \in \clb(H^2(\D))$ is said to be $S$-Toeplitz if
\[
S^* T S = T.
\]
We quickly observe that symmetric Toeplitz operators are necessarily $S$-Toeplitz (see Proposition \ref{thm: symm S Toep}): If $C$ is a conjugation on $H^2(\D)$, and
\[
f_n:= C z^n \qquad (n \in \Z_+),
\]
then
\[
S:= C M_z C,
\]
is a shift corresponding to the orthonormal basis $\{f_n\}_{n\in \Z_+}$ of $H^2(\D)$. Moreover, if $T_\vp$, $\vp \in L^\infty(\T)$, is $C$-symmetric, then $T_\vp$ is $S$-Toeplitz.

However, the $S$-Toeplitz condition appears to be not sufficient to maintain the symmetricity of Toeplitz operators. To remedy this situation, we introduce canonical factorizations of conjugations which is a careful refinement of factorizations of unitary operators by Godi\v{c} and Lucenko \cite{Godic}, and  Garcia and Putinar \cite{garcia2}. We prove that for a conjugation $C$ on $H^2(\D)$, there is a unique unitary $U \in \clb(H^2(\D))$ such that
\[
C = U J_{H^2(\D)},
\]
which we call the \textit{canonical factorization} of $C$ (see Proposition \ref{Prop: C = UC =CU*} and Definition \ref{def: con factor}). Here
\[
J_{H^2(\D)} \Big(\sum_{n=0}^{\infty} a_n z^n\Big) = \sum_{n=0}^{\infty} \bar{a}_n z^n \qquad (\sum_{n=0}^{\infty} a_n z^n \in H^2(\D)),
\]
is the \textit{canonical conjugation} on $H^2(\D)$. Finally, in Theorem \ref{thm: S toep 1}, we connect symmetricity of $T_\vp$, $\vp \in L^\infty(\T)$, with formal Toeplitz matrices: Let $C \in \clc(H^2(\D))$. Then $T_\vp$ is $C$-symmetric if and only if
\[
[T_\vp]_{\{f_n\}_{n \in \Z_+}} = [T_\vp]_{\{z^n\}_{n \in \Z_+}}^t,
\]
where $f_n:= Uz^n = C z^n$, $n \in \Z_+$, and $C = U J_{H^2(\D)}$ is the canonical factorization of $C$. Since $[T_\vp]_{\{z^n\}_{n \in \Z_+}}^t$ is also a formal Toeplitz matrix, the above equality in particular implies that $T_\vp$ is $S$-Toeplitz. In other words:

\begin{theorem}\label{thm:intro 1}
$T_\vp$ is $C$-symmetric if and only if $T_\vp$ is $S$-Toeplitz and
\[
[T_\vp]_{\{f_n\}_{n \in \Z_+}} = [T_\vp]_{\{z^n\}_{n \in \Z_+}}^t.
\]
\end{theorem}

Here $[T]_{\{g_n\}_{n \in \Z_+}}$ denotes the formal matrix representation of $T \in \clb(H^2(\D))$ with respect to a given orthonormal basis $\{g_n\}_{n \in \Z_+}$ of $H^2(\D)$. For instance, if $\vp  = \sum_{n=-\infty}^{\infty} \vp_n z^n \in L^\infty(\T)$, then we have
\[
[T_\vp]_{\{z^n\}_{n \in \Z_+}} =
\begin{bmatrix}	\vp_0 & \vp_{-1} & \vp_{-2} & \vp_{-3} & \hdots
\\
\vp_{1} & \vp_0 & \vp_{-1} & \vp_{-2} & \ddots
\\
\vp_{2} & \vp_{1} & \vp_0 & \vp_{-1} & \ddots
\\
\vp_{3} & \vp_{2} & \vp_{1} & \vp_0 & \ddots
\\
\vdots&\vdots&\vdots&\vdots&\ddots
\end{bmatrix},
\]
the familiar Toeplitz matrix representation of the Toeplitz operator $T_\vp$. We believe that the perspective of $S$-Toeplitz operators in the theory of the symmetric operators is completely new.

Our second approach to the characterization of symmetric Toeplitz operators follows the line of existing routes and substantially improves and unifies all the known partial results. We make use of coordinate-free representations of conjugations. More specifically, given a Hilbert space $\clh$, we denote by $B_{\clh}$ the set of all ordered orthonormal bases of $\clh$. We also denote by ($\clb_a(\clh)$) $\cll_a(\clh)$ the space of (bounded) anti-linear operators on $\clh$. The following is our second classification of conjugations (see Proposition \ref{conju_Charac_theorem}), which also unifies all the existing results about representations of conjugations: Let $C \in \cll_a(\clh)$. Then $C \in \clc(\clh)$ if and only if there exists $\{g_n\}_{n \in \Lambda} \in B_{\clh}$ (in fact, $C g_n = g_n$ for all $n \in \Z_+$) such that
\[
C (\sum_{n} a_n\tau_n) = \sum_{n} \sum_{m}\bar{a}_n c_{n,m}^{(\tau)} \tau_m,
\]
for all $\tau = \{\tau_n\}_{n \in \Lambda} \in B_{\clh}$ and $\sum_{n} a_n\tau_n \in \clh$, $a_n \in \C$, where
\[
c_{n,m}^{(\tau)}=\displaystyle\sum_{k}\langle g_k,\tau_n\rangle \langle g_k,\tau_m\rangle \qquad (m, n \in \Lambda).
\]
It is evident that $|c_{n,m}^{(\tau)}| \leq 1$ and
\[
c_{n,m}^{(\tau)} = c_{m,n}^{(\tau)} \qquad (m, n \in \Lambda).
\]
With this classification in hand, we again turn to Toeplitz operators on $H^2(\D)$. We first fix the canonical basis of $H^2(\D)$ as
\[
\zeta := \{z^n\}_{n \geq 0} \in B_{H^2(\D)}.
\]
The following summarizes our characterizations of symmetric Toeplitz operators:

\begin{theorem}\label{thm: summary}
Let $C$ be a conjugation on $H^2(\D)$. Suppose $C =U J_{H^2(\D)}$ is the canonical factorization of $C$. Define the unilateral shift $S$ on $H^2(\D)$ by $S := CM_zC$. Let
\[
u_{n,m}= \langle U z^n, z^m \rangle,
\]
and
\[
f_n:= Uz^n,
\]
for all $m, n \in \Z_+$. If $\vp = \sum_{n=-\infty}^{\infty} \vp_n z^n \in L^\infty(\T)$, then the following are equivalent:
\begin{enumerate}
\item $T_\vp$ is $C$-symmetric.
\item $T_\vp$ is $S$-Toeplitz and
\[
[T_\vp]_{\{f_n\}_{n \in \Z_+}} = [T_\vp]_{\{z^n\}_{n \in \Z_+}}^t.
\]
\item $U^*T_{\vp}U$ is a Toeplitz operator and
\[
[U^*T_{\vp}U]_{\{z^n\}_{n \in \Z_+}} = [T_{\vp}]^{t}_{\{z^n\}_{n \in \Z_+}}.
\]
\item For all $m, n \in \Z_+$, we have
\[
\vp_{m-n} = \sum_{i,j} u_{m,i} {\vp}_{i-j}\overline{u_{j,n}}.
\]
\item For all $j,k\in \Z_+$, we have
\[
\sum_{n=0}^{\infty}\overline{\vp_{n-k}}c_{n,j}^{(\zeta)} = \sum_{n=1}^{\infty}\overline{\vp_n} c_{k,n+j}^{(\zeta)} + \sum_{l=0}^{j}\overline{\vp_{-l}} c_{k,j-1}^{(\zeta)},
\]
where  
\[
c_{n,m}^{(\zeta)} = \displaystyle\sum_{k}\langle g_k, z^n \rangle \langle g_k, z^m \rangle \qquad (m, n\in \Z_+),
\]
and $C g_p = g_p$, $p \in \Z_+$, for some $\{g_n\}_{n \in \Z_+} \in B_{H^2(\D)}$.

\item For all $j,k\in \Z_+$, we have
\[
\sum_{n=0}^{\infty}\overline{\vp_{n-k}}u_{n,j} = \sum_{n=1}^{\infty}\overline{\vp_n}u_{k,n+j} + \sum_{l=0}^{j}\overline{\vp_{-l}}u_{k,j-l}.
\]
\end{enumerate}
\end{theorem}

In addition to the above and following the wish list of \cite[Section 10]{garcia1}, we also present results on symmetric operators in several variables. We first introduce $S$-Toeplitz operators in several variables and then prove similar characterizations of symmetric Toeplitz operators on the Hardy space over the open unit polydisc in $\mathbb{C}^n$.

It is worth pointing out that there have been many attempts to provide (partial) answers to Question \ref{Q1} (for instance, see \cite{Bu, Ko_Lee_SCTO, Li}). However, our approach and objective are somehow different. As already pointed out, the key to our analysis is concrete representations and canonical factorizations of conjugations. Our answer to Question \ref{Q1} unifies all known partial results in the literature.

Moreover, our approach yields a new characterizations of complex symmetric operators: Given a basis $\{e_n\}_{n \in \Lambda} \in B_\clh$, we define the conjugation $J_{\clh}$ on $\clh$ by (see Definition \ref{eqn: can conjug})
\[
J_{\clh} (\sum_n a_n e_n) = \sum_n \bar{a}_n e_n,
\]
for all $\sum_n a_n e_n \in \clh$. In this setting, if $C$ is a conjugation on $\clh$, then there is a unique unitary $U \in \clb(\clh)$ such that
\begin{equation}\label{eqn: intro c uj}
C=U J_\clh.
\end{equation}
Theorem \ref{thm: charac symm op} then states:

\begin{theorem}
Let $\clh$ be a Hilbert space, $\{e_n\}_{n\in \Lambda} \in B_\clh$, and let $C\in \clc(\clh)$. Then $T\in \mathcal{B}(\clh)$ is $C$-symmetric if and only if
\[
[T]_{\{e_n\}_{n\in \Lambda}}^t = [U^*TU]_{\{e_n\}_{n\in \Lambda}},
\]
where $C=UJ_\clh$ is as in \eqref{eqn: intro c uj}.
\end{theorem}

The rest of the paper is organized as follows. In the following section, we will discuss a method of factorizations of conjugations. Along the way, we will introduce the necessary terminology and record some observations that will be useful in the sequel. Here we also present a pair of characterizations of conjugations.

In Section \ref{sec: rep of conj}, we present our third and final characterization of conjugations. Such characterizations essentially generalize and unify all the existing results concerning representations of conjugations.

Section \ref{sec: S Toeplitz} deals with the notion of $S$-Toeplitz operators. We prove that a symmetric Toeplitz operator is necessarily $S$-Toeplitz. The converse, however, does not hold in general. Section \ref{sec: S Toeplitz and symm} identifies the missing link and proves that the converse holds if the matrix representation of the Toeplitz operator corresponding to the basis of the shift $S$ equals the transpose of the matrix representation of the given Toeplitz operator.
	
In Section \ref{sec: Conj and Toeplitz}, we present our final characterization of symmetric Toeplitz operators. Here we follow the analysis of Section \ref{sec: rep of conj}. The key is the representations of conjugations on the Hardy space with respect to the canonical basis. Section \ref{sec: polydisc} classifies symmetric Toeplitz operators on the Hardy space over the unit polydisc. One of the keys is the notion of $S$-Toeplitz operators in several variables.

In Section \ref{sec: composition}, we connect complex symmetric Toeplitz operators and a class of composition operators. We construct a conjugation via a unitary weighted composition operator and discuss the symmetricity of Toeplitz operators corresponding to composition-based conjugations.

Section \ref{sec: examples} consists of more assorted examples of symmetric Toeplitz operators. The appendix, Section \ref{sec: appendix}, at the end of the paper contains some results on intertwiners that are not directly related to Toeplitz operators, but fit well in the context of symmetric operators and may be of independent interest.

\newsection{Factorizations of conjugations}

In this section, we describe natural methods of factorizations of conjugations on Hilbert spaces. This will be a key tool in our first characterizations of symmetric Toeplitz operators. Some of the results of this section may be of independent interest and may have other applications.

We begin with some notation. Given a Hilbert space $\clh$, we fix an element $\{e_n\}_{n\in \Lambda} \in B_{\clh}$, and we call it the \textit{canonical basis} of $\clh$. Depending on the dimension of $\clh$, the index set $\Lambda$ is either a finite set or a countably infinite set. Recall also that $B_{\clh}$ is the set of all ordered orthonormal bases of $\clh$. The choice of the canonical basis $\{e_n\}_{n\in \Lambda} \in B_{\clh}$ might depend on the class of operators under consideration. For instance, in the context of Toeplitz operators on $H^2(\D)$, we set, by convention
\[
\zeta = \{z^n\}_{n\in \Z_+} \in B_{H^2(\D)},
\]
the canonical basis of $H^2(\D)$.

\begin{definition}\label{eqn: can conjug}
Let $\clh$ be a Hilbert space. Suppose $\{e_n\}_{n\in \Z_+} \in B_{\clh}$ is the canonical basis of $\clh$. The canonical conjugation of $\clh$ is the conjugation $J_{\clh}$ defined by
\[
J_{\clh} (\sum_n a_n e_n) = \sum_n \bar{a}_n e_n,
\]
for all $\sum_n a_n e_n \in \clh$.
\end{definition}

In the case that $\clh =H^2(\D)$, it can be easily proved that
\begin{equation}\label{eqn: CH Mz comm}
J_{H^2(\D)} M_z = M_z J_{H^2(\D)},
\end{equation}
where $M_z$ denotes the multiplication operator by the coordinate function $z$ on $H^2(\D)$, that is
\[
(M_z f)(w) = w f(w) \qquad (w \in \D).
\]
Also note that in view of the polarization identity, a map $C \in \clb_a(\clh)$ is isometric if and only if (cf. \cite[Page 4]{garcia_prodan_putinar})
\[
\langle Cx, Cy \rangle = \langle y, x \rangle \qquad (x, y\in \clh).
\]
In particular, if $C$ is a conjugation, then
\[
\begin{split}
\langle Cx, y \rangle & = \langle Cx, C^2y \rangle
\\
& = \langle C x, C(Cy) \rangle
\\
& = \langle Cy, x \rangle,
\end{split}
\]
and hence
\begin{equation}\label{eqn: conj inner prod}
\langle Cx, y \rangle = \langle C y, x \rangle \qquad (x, y \in \clh).
\end{equation}

We also need one of the most elementary facts about conjugations \cite[Lemma 1]{garcia1}:

\begin{lemma}\label{conjuONB}
Let $C \in \cll_a(\clh)$. Then $C$ is a conjugation if and only if there exists $\{f_n\}_{n \in \Lambda} \in B_{\clh}$ such that $C f_n = f_n$ for all $n \in \Lambda$.
\end{lemma}

Recall that $\cll_a(\clh)$ ($\clb_a(\clh)$) denotes the space of all anti-linear (bounded anti-linear) operators on $\clh$.

In view of the above lemma, we introduce:

\begin{definition}\label{def: conj corres}
Let $C \in \clc(\clh)$. We say that $C$ is a conjugation corresponding to $\{f_n\}_{n\in \Lambda} \in B_{\clh}$ if $C f_n = f_n$ for all $n \in \Lambda$.
\end{definition}

Up to (linear) unitary equivalence, canonical conjugation is the only conjugation on a Hilbert space:

\begin{lemma}\label{lemma1}
Let $C \in \cll_a(\clh)$. Then $C$ is a conjugation if and only if there exists a unitary $U \in \clb(\clh)$ such that $C = U^* J_{\clh} U$.
\end{lemma}
\begin{proof}
If $C = U^* J_{\clh} U$, then it is easy to see that $C$ satisfies all the conditions of conjugations. For the reverse direction, by Lemma \ref{conjuONB}, there exists $\{f_n\} \in B_{\clh}$ such that $C f_n = f_n$ for all $n$. For each $x\in \clh$, we know that $x = \sum_n \langle x, f_n\rangle f_n$. Then
\[
Ux= \sum_n \langle x, f_n\rangle e_n \qquad (x \in \clh),
\]
defines a unitary $U \in \clb(\clh)$. It is now easy to check that $J_{\clh} U = UC$.
\end{proof}

The same is true up to anti-linear unitary equivalence. In other words, if we define $U$ on $\clh$ by
\[
Ux = \sum_n \langle f_n, x \rangle e_n \qquad (x \in \clh),
\]
then $U$ is an anti-unitary operator, and hence $UC = J_{\clh} U$, that is, $C$ is anti-unitarily equivalent to $J_{\clh}$. This and Lemma \ref{lemma1} clearly imply that all conjugations on a Hilbert space are unitarily equivalent:

\begin{corollary}\label{cor: all conj are same}
Conjugations are unitarily as well as anti-unitarily equivalent.
\end{corollary}

We now justify the canonicity of canonical conjugations, which also yields the first characterization of conjugations in this paper.

\begin{proposition}\label{Prop: C = UC =CU*}
Let $\clh$ be a Hilbert space, and let $C \in \cll_a(\clh)$. Then $C$ is a conjugation if and only if there is a unique unitary $U$ on $\clh$ such that
\[
C = U J_{\clh} = J_{\clh} U^*.
\]
\end{proposition}
\begin{proof}
Let $C \in \cll_a(\clh)$. If $C = U J_{\clh} = J_{\clh} U^*$ for some unitary $U \in \clb(\clh)$, then clearly $C$ is a conjugation. On the other hand, if $C$ is a conjugation, then the required equality follows from the fact that $U:= C J_{\clh}$ is a unitary on $\clh$.

\noindent If $\tilde U \in \clb(\clh)$ is a unitary such that $C = \tilde U J_{\clh}$, then $\tilde U J_{\clh} = U J_{\clh}$, and hence $\tilde U = U$. This proves the uniqueness part and completes the proof of the proposition.
\end{proof}

This observation is essentially a refinement of the classical factorizations of unitaries by Godi\v{c} and Lucenko \cite{Godic} (also see Garcia and Putinar \cite[Lemma 1]{garcia2}).

The above result motivates us to introduce canonical factorizations of conjugations.

\begin{definition}\label{def: con factor}
Let $\clh$ be a Hilbert space, and let $C \in \clc(\clh)$. Then
\[
C = U J_{\clh},
\]
is called the canonical factorization of $C$, where $U \in \clb(\clh)$ is a unitary.
\end{definition}

The unitary part of the canonical factorization of the conjugation $C$ enjoys a rather special property, namely
\[
\langle U e_n, e_m \rangle = \langle U e_m, e_n \rangle \qquad (m, n \in \Lambda),
\]
that is, $U$ is symmetric with respect to the canonical basis.

\begin{proposition}\label{prop: U is symmetric}
Let $U \in \clb(\clh)$ be a unitary. Then $C = U J_{\clh}$ defines a conjugation on $\clh$ if and only if $U$ is $J_{\clh}$-symmetric.
\end{proposition}
\begin{proof}
Suppose $C = U J_{\clh}$ is a conjugation. For each $m, n \in \Lambda$, we have
\[
\langle U e_n, e_m \rangle = \langle U J_{\clh} e_n, e_m \rangle = \langle C e_n, e_m \rangle = \langle C e_m, e_n \rangle,
\]
where the latter equality follows from \eqref{eqn: conj inner prod}. By reversing the argument, it follows that $\langle U e_n, e_m \rangle = \langle U e_m, e_n \rangle$. For the converse, suppose $\langle U e_n, e_m \rangle = \langle U e_m, e_n \rangle$ for all $m, n \in \Lambda$. We need to prove that $C := U J_{\clh}$ is a conjugation on $\clh$. Clearly, $C$ is anti-linear and isometry. Moreover, for each $m, n \in \Lambda$, we have
\[
\begin{split}
\langle U J_{\clh} e_n, e_m \rangle & = \langle U e_n, e_m \rangle
\\
& = \langle U e_m, e_n \rangle
\\
& = \langle e_m, U^* e_n \rangle
\\
& = \langle J_{\clh} e_m, U^* e_n \rangle
\\
& = \langle J_{\clh} U^* e_n,  e_m \rangle,
\end{split}
\]
as $J_{\clh}$ is a conjugation. Therefore, $C = U J_{\clh} = J_{\clh} U^*$, and hence
\[
C^2 = (U J_{\clh}) (J_{\clh} U^*) = U U^* = I.
\]
Consequently, $C$ is a conjugation.
\end{proof}

The above yields our second characterization of conjugations.

\newsection{Representations of conjugations}\label{sec: rep of conj}

This section presents our third and final characterization of conjugations, which also yields useful representations of conjugations. We will also illustrate how one can recover the commonly used conjugations from our representations of conjugations (cf. Examples \ref{illustration_example} and \ref{example Ferreira}).

Given a Hilbert space $\clh$ and a fixed basis $\{e_n\}_{n \in \Lambda} \in B_{\clh}$, we define
\[
c_{n,m}^{(\tau)} = \sum_{k} \langle e_k, \tau_n \rangle \langle e_k, \tau_m \rangle \qquad (m, n \in \Lambda),
\]
for all $\tau = \{\tau_n\}_{n \in \Lambda} \in B_{\clh}$. Usually, the basis $\{e_n\}_{n \in \Lambda} \in B_{\clh}$ will be clear from the context and we do not include it in the above notation. We are now ready for our third characterization of conjugations.

\begin{proposition}\label{conju_Charac_theorem}
Let $C \in \cll_a(\clh)$. Then $C$ is a conjugation if and only if there exists $\{f_n\}_{n \in \Lambda} \in B_{\clh}$ such that
\[
C(\sum_{n} a_n\tau_n) = \sum_{n} \sum_{m} \bar{a}_n c_{n,m}^{(\tau)} \tau_m,
\]
for all $\tau = \{\tau_n\}_{n \in \Lambda} \in B_{\clh}$ and $\{a_n\}_{n \in \Lambda} \in \ell^2$, where
\[
c_{n,m}^{(\tau)}=\displaystyle\sum_{k}\langle f_k,\tau_n\rangle \langle f_k,\tau_m\rangle \qquad (m, n \in \Lambda).
\]
\end{proposition}
\begin{proof}
Let $C$ be a conjugation on $\clh$. By Lemma \ref{conjuONB}, there exists $\{f_{n}\}_{n \in \Lambda} \in B_{\clh}$ such that $Cf_{n} = f_{n}$ for all $n \in \Lambda$. Fix a basis $\{\tau_{n}\}_{n \in \Lambda} \in B_{\clh}$. For each $m, n \in \Lambda$, we have
\[
\tau_{m} = \sum_{j} \inp{\tau_{m}}{f_j} f_j,
\]
and
\[
\begin{split}
C \tau_{n} & = \sum_{k} \inp{C\tau_{n}}{f_k}f_k
\\
& = \sum_{k} \inp{C f_{k}}{\tau_{n}} f_k
\\
& = \sum_{k} \inp{f_{k}}{\tau_{n}} f_k.
\end{split}
\]
In the above, we have used the fact that $\inp{Cf}{g} = \inp{Cg}{f}$ for all $f, g \in \clh$ (see \eqref{eqn: conj inner prod}). Therefore
\[
\begin{split}
\inp{C \tau_n}{\tau_m} & = \inp{\sum_{k} \inp{f_{k}}{\tau_{n}} f_k}{\sum_{j} \inp{\tau_{m}}{f_j} f_j}
\\
& = \sum_{k}\inp{f_{k}}{\tau_{n}}\overline{\inp{ \tau_{m}}{f_k}}
\\
& = \sum_{k}\inp{f_{k}}{\tau_{n}}\inp{ f_k}{\tau_{m}}
\\
& = c^{(\tau)}_{n,m},
\end{split}
\]
and finally
\[
\begin{split}
C(\sum_{n} a_n\tau_n) & = \sum_{n} \bar{a}_n C\tau_{n}
\\
& = \sum_{n} \sum_{m}\bar{a}_n \inp{C\tau_{n}}{\tau_{m}}\tau_m
\\
& = \sum_{n} \sum_{m}\bar{a}_n c^{(\tau)}_{n,m}\tau_m.
\end{split}
\]
To show the converse, we choose, in particular, that $\tau_n=f_n$ for all $n \in \Lambda$. Then
\[
c_{n,m}^{(\tau)} =
\begin{cases}
1 & \mbox{if } m = n
\\
0 & \mbox{otherwise},
\end{cases}
\]
which yields
\[
C(\sum_{n} a_n\tau_n) = C(\sum_{n} a_n f_n) = \sum_{n} \sum_{m}\bar{a}_n c_{n,m}^{(\tau)} f_m = \sum_{n} \bar{a}_n f_n,
\]
that is, $C(\sum_{n} a_n\tau_n) = \sum_{n} \bar{a}_n\tau_n$ for all $\sum_{n} a_n\tau_n \in \clh$. This proves that $C$ is a conjugation and completes the proof of the proposition.
\end{proof}

The above proposition should be compared with \cite[Proposition 6]{ferreira_junior}, which gives representations of conjugations on $H^2(\D)$ with non-explicit coefficients. In the present case, the result is complete in the sense that it holds for general conjugations, and the coefficients $\{c_{n,m}^{(\tau)}\}$ are explicit and completely determined by the basis $\{f_n\}_{n \in \Lambda} \in B_{\clh}$ corresponding to $C$ and arbitrary basis $\tau \in B_{\clh}$.

We also observe, in view of Proposition \ref{conju_Charac_theorem}, that if $\clh$ is infinite-dimensional Hilbert space and $C \in \clc(\clh)$, then necessarily
\[
\{c_{n,m}^{(\tau)}\}_{m \in \Lambda} \in \ell^2,
\]
for all $n \in \Lambda$ and $\tau \in B_{\clh}$. The necessary part of the above proposition yields representations of conjugations which we record for future references:

\begin{corollary}\label{cor: rep of C}
Let $\tau = \{\tau_n\}_{n \in \Lambda} \in B_{\clh}$, and let $C$ be a conjugation on $\clh$ corresponding to $\{f_n\}_{n \in \Lambda} \in B_{\clh}$. Then
\[
C(\sum_{n} a_n\tau_n) = \sum_{n} \sum_{m} \bar{a}_n c_{n,m}^{(\tau)} \tau_m,
\]
where
\[
c_{m,n}^{(\tau)}=\displaystyle\sum_{k}\langle f_k,\tau_n\rangle \langle f_k,\tau_m\rangle,
\]
for all $m, n \in \Lambda$.
\end{corollary}

The following example illustrates importing representations of conjugations with respect to a suitable basis ($\{z^n\}_{n \in \Z_+} \in B_{H^2(\D)}$ in this case).

\begin{example}\label{illustration_example}
Fix $\theta, \xi \in \R$, and consider the basis $\{f_n\}_{n \in \Z_+} \in B_{H^2(\D)}$, where
\[
f_n= e^{\frac{i\xi}{2}} e^{\frac{-i n \theta}{2}}z^n \qquad (n \in \Z_+).
\]
Let $C_{\theta, \xi}$ denote the conjugation corresponding to $\{f_n\}_{n \in \Z_+} \in B_{H^2(\D)}$, that is, $C_{\theta, \xi} f_m = f_m$ for all $m \in \Z_+$. Now we consider the canonical basis $\zeta = \{z^n\}_{n \in \Z_+} \in B_{H^2(\D)}$. We compute
\[
\begin{split}
c_{n,m}^{(\zeta)}&=\displaystyle\sum_{k=0}^{\infty}\langle f_k,z^n\rangle \langle f_k, z^m\rangle
\\
& =
\begin{cases}
0 & \text{if } m\neq n \\
e^{i\xi}e^{-in\theta} & \text{if } m=n.
\end{cases}
\end{split}
\]
Consequently, for each $f = \sum_{n=0}^\infty a_nz^n \in H^2(\D)$, we have
\[
\begin{split}
C_{\theta, \xi}(\sum_{n=0}^\infty a_nz^n)& = \sum_{n=0}^\infty \sum_{m=0}^\infty \bar{a}_n c_{n, m}^{(\zeta)} z^m
\\
& = \sum_{n=0}^\infty \bar{a}_n c_{n, n}^{(\zeta)} z^n
\\
& = e^{i\xi} \sum_{n=0}^\infty \bar{a}_n e^{-in\theta} z^n,
\end{split}
\]
and hence
\[
(C_{\theta, \xi} f)(z)=e^{i\xi}\overline{f(e^{i\theta}\bar{z})} \qquad (f \in H^2(\D)).
\]
In particular, if $\theta = \xi = 0$, then we get back the canonical conjugation $J_{H^2(\D)}$ of $H^2(\D)$.
\end{example}

The conjugation $C_{\theta, \xi}$ on $H^2(\D)$ was introduced in \cite{Ko_Lee_SCTO}. The above example asserts that the representation of $C_{\theta, \xi}$ can be fully recovered from our general approach. The following example is another instance \cite[Proposition 2.6]{Li}:

\begin{example}\label{example Ferreira}
Let $\{\alpha_n\}_{n \in \Z_+} \subseteq \T$, and suppose
\[
f_n= \alpha_n z^n \qquad (n \in \Z_+).
\]
Clearly, $\{f_n\}_{n \in \Z_+} \in B_{H^2(\D)}$. Let $C_{\alpha}$ denote the conjugation corresponding to $\{f_n\}_{n \in \Z_+} \in B_{H^2(\D)}$, that is, $C_{\alpha} f_n = f_n$ for all $n \in \Z_+$. As in the above example, with the canonical basis $\zeta = \{z^n\}_{n \in \Z_+} \in B_{H^2(\D)}$, we have
\[
c_{n,m}^{(\zeta)} =
\begin{cases}
0 & \text{if } m\neq n \\
{\alpha_n}^2 & \text{if } m=n.
\end{cases}
\]
A similar computation leads to the representation of $C_\alpha$ as
\[
C_{\alpha}(\sum_{n=0}^\infty a_nz^n) = \sum_{n=0}^\infty \bar{a}_n {\alpha_n}^2 z^n,
\]
for all $\sum_{n=0}^\infty a_nz^n \in H^2(\D)$.
\end{example}

Clearly, Example \ref{illustration_example} follows from the above example with $\alpha_n = e^{\frac{i\xi}{2}} e^{\frac{-i n \theta}{2}}$ for all $n \in \Z_+$. We shall return to this theme in Section \ref{sec: S Toeplitz and symm}, where we will present one of the characterizations of symmetric Toeplitz operators.

\section{$S$-Toeplitz operators}\label{sec: S Toeplitz}

This short section aims to introduce the notion of $S$-Toeplitz operators and signal its role to symmetric Toeplitz operators. From now onwards, all Hilbert spaces are assumed to be of infinite-dimensional, unless specified otherwise or clear from the context.

We begin with the definition of shift operators, and refer the reader to the monograph \cite{RR-Book} for a more detailed treatment of shift operators. An operator $S \in \clb(\clh)$ is called \textit{shift} if $S$ is an isometry (that is, $S^*S = I_{\clh}$) and $S$ is pure (that is, $\|S^{*m}h\| \raro 0$ for all $h \in \clh$). The multiplicity of a shift $S$ is the number
\[
mult(S) := \text{dim} (\ker S^*) \in \mathbb{N} \cup \{\infty\}.
\]
In view of the generating wandering subspace property, an operator $S$ on $\clh$ is a shift of multiplicity one if and only if there exists $\{f_n\}_{n \in \Z_+} \in B_{\clh}$ such that
\[
S f_n = f_{n+1} \qquad (n \in \Z_+).
\]
Shifts of multiplicity one are commonly known as unilateral shift:

\begin{definition}\label{def: shift}
We say that $S \in \clb(\clh)$ is a unilateral shift corresponding to $\{f_n\}_{n \in \Z_+} \in B_{\clh}$ if $S f_n = f_{n+1}$ for all $n \in \Z_+$.
\end{definition}

Clearly, $M_z$ on $H^2(\D)$ is a unilateral shift corresponding to the canonical basis $\{z^n\}_{n \in \Z_+} \in B_{H^2(\D)}$. Recall that an operator $T \in \clb(H^2(\D))$ is a Toeplitz operator if and only if (see the monographs \cite{RGD, RR-Book})
\[
M_z^* T M_z = T.
\]
With this motivation in mind, we now introduce $S$-Toeplitz operators which also include all the classical Toeplitz operators (see \cite[Chapter 3]{RR-Book}).

\begin{definition}
Let $S$ be a shift on $\clh$. An operator $T \in \clb(\clh)$ is called $S$-Toeplitz if
\[
S^* T S = T.
\]
\end{definition}

Recall that for a Hilbert space $\clh$ with the canonical basis $\{e_n\}_{n \in \Z_+} \in B_{\clh}$, the canonical conjugation $J_{\clh} \in \clc(\clh)$ is defined by (see Definition \ref{eqn: can conjug})
\[
J_{\clh}(\sum_{n =0}^\infty a_n e_n) = \sum_{n =0}^\infty \bar{a}_n e_n,
\]
for all $\sum_{n =0}^\infty a_n e_n \in \clh$. We define the \textit{canonical shift} $S_{\clh} \in \clb(\clh)$ by
\[
S_{\clh} e_n = e_{n+1} \qquad (n \in \Z_+).
\]
Then, as in \eqref{eqn: CH Mz comm}, it follows that
\[
J_{\clh} S_{\clh} = S_{\clh} J_{\clh}.
\]

Given a conjugation on a Hilbert space, there is a natural way to construct a unilateral shift on the same Hilbert space:

\begin{lemma}\label{lemma: C = UC =CU*}
Let $C \in \clc(\clh)$. Suppose $C = U J_{\clh}$ is the canonical factorization of $C$, and $f_n:=U e_n$ for all $n \in \Z_+$. Then
\[
S:= C S_{\clh} C,
\]
is a unilateral shift corresponding to $\{f_n\} \in B_{\clh}$.
\end{lemma}
\begin{proof}
Since $\{e_n\}_{n \in \Z_+} \in B_{\clh}$ is the canonical basis of $\clh$, and $C = U J_{\clh}$, it follows that
\[
C e_n = U J_{\clh} e_n = U e_n,
\]
that is
\begin{equation}\label{eqn: C en = B en}
C e_n = U e_n \qquad (n \in \Z_+).
\end{equation}
By the definition, we then have $f_n = U e_n = C e_n$, and consequently
\[
\begin{split}
S f_n & = C S_{\clh} C (C e_n)
\\
& = CS_{\clh} e_n
\\
& = C e_{n+1}
\\
& = U e_{n+1}
\\
& = f_{n+1},
\end{split}
\]
for all $n \in \Z_+$. Now the conclusion follows from the fact that $S$ is a linear isometry.
\end{proof}

Specializing the above lemma to the case that $\clh$ is the Hardy space we conclude:

\begin{corollary}\label{cor: CMzC shift}
If $C$ is a conjugation on $H^2(\D)$, then $CM_zC$ is a unilateral shift on $H^2(\D)$.
\end{corollary}

Our entry point to address Question \ref{Q1} is that symmetric Toeplitz operators are $S$-Toeplitz.

\begin{proposition}\label{thm: symm S Toep}
Let $\varphi\in L^\infty(\T)$, $C$ a conjugation on $H^2(\D)$, and suppose $S = CM_zC$. If $T_\vp$ is $C$-symmetric, then $T_{\vp}$ is $S$-Toeplitz.
\end{proposition}
\begin{proof}
By Corollary \ref{cor: CMzC shift}, we know that $S = CM_zC$ is a unilateral shift. Moreover, $CT_\varphi C=T^*_\varphi$, by definition, and $M_z^*T_\varphi M_z=T_\varphi$, as $T_\varphi$ is a Toeplitz operator. Then
\[
\begin{split}
T_\varphi  & = C T^*_\varphi C
\\
& = CM_z^* T^*_\varphi M_z C
\\
& = CM_z^*CT_\varphi C M_z C
\\
& = S^*T_\varphi S,
\end{split}
\]
as $S^* = C M_z^* C$. This completes the proof of the proposition.
\end{proof}

The converse is not true in general. Before we present a counterexample, we recall a few facts concerning the Hardy space. Denote by $H^\infty(\D)$ the Banach algebra of all bounded analytic functions on $\D$. Given $\theta \in H^\infty(\D)$, denote by $M_\theta$ the multiplication operator on $H^2(\D)$. That is
\[
M_\theta f = \theta f \qquad (f \in H^2(\D)).
\]
It then follows that for $X \in \clb(H^2(\D))$, that $X = M_\theta$ for some $\theta \in H^\infty(\D)$ if and only if
\[
X M_z = M_z X.
\]
Equivalently
\[
\{M_z\}' = \{M_\theta: \theta \in H^\infty(\D)\}.
\]
Also note that $T_\theta = M_\theta$ for all $\theta \in H^\infty(\D)$. Now we turn to the counterexample. Let $\varphi \in H^\infty(\D)$ be a nonconstant function, and let
\[
S = J_{H^2(\D)} M_z J_{H^2(\D)}.
\]
Since $J_{H^2(\D)} M_z = M_z J_{H^2(\D)}$ (see \eqref{eqn: CH Mz comm}), it follows that $S = M_z$, and hence $S^*T_\varphi S=T_\varphi$. However, $T_{\vp}$ is not $C$-symmetric for any conjugation $C$ on $H^2(\D)$ (see \cite[Corollary 2.2]{Ko_Lee_SCTO} and \cite[Proposition 2.2]{Noor}).

The following section will furnish the missing link that would unlock the complete classification of $C$-symmetric Toeplitz operators.

\section{Symmetric and $S$-Toeplitz operators}\label{sec: S Toeplitz and symm}

In this section, in continuation of Proposition \ref{thm: symm S Toep}, we present our first characterization of symmetric Toeplitz operators. Our answer connects symmetric Toeplitz operators with the classical $S$-Toeplitz operators. Let
\[
\vp = \sum_{n=-\infty}^{\infty} \vp_n z^n \in L^\infty(\T).
\]
Denote by $[T_\vp]_{\{z^n\}_{n \in \Z_+}}$ the formal matrix representation of the Toeplitz operator $T_\vp$ with respect to the canonical basis $\{z^n\}_{n \in \Z_+} \in B_{H^2(\D)}$. Therefore, we have the familiar Toeplitz matrix representation
\[
[T_\vp]_{\{z^n\}_{n \in \Z_+}} =
\begin{bmatrix}	\vp_0 & \vp_{-1} & \vp_{-2} & \vp_{-3} & \hdots
\\
\vp_{1} & \vp_0 & \vp_{-1} & \vp_{-2} & \ddots
\\
\vp_{2} & \vp_{1} & \vp_0 & \vp_{-1} & \ddots
\\
\vp_{3} & \vp_{2} & \vp_{1} & \vp_0 & \ddots
\\
\vdots&\vdots&\vdots&\vdots&\ddots
\end{bmatrix}.
\]
For each $p, q \in \Z_+$, observe that
\[
\begin{split}
\langle T_\vp z^p, z^q \rangle & = \langle P_{H^2(\D)} L_\vp z^p, z^q \rangle
\\
& = \langle \vp z^p, z^q \rangle
\\
& = \langle \sum_{n=-\infty}^{\infty} \vp_n z^{n+p}, z^q \rangle
\\
& = \vp_{q-p},
\end{split}
\]
that is
\begin{equation}\label{eqn: Toepl coeff}
\langle T_\vp z^p, z^q \rangle = \vp_{q-p} \qquad (p, q \in \Z_+).
\end{equation}
Now, suppose $\{f_n\}_{n \in \Z_+} \in B_{H^2(\D)}$. Denote by $S \in \clb(H^2(\D))$ the unilateral shift corresponding to $\{f_n\}_{n \in \Z_+} \in B_{H^2(\D)}$, that is (see Definition \ref{def: shift})
\[
S f_n = f_{n+1} \qquad (n \in \Z_+).
\]
Again, denote by $[T_\vp]_{\{f_n\}_{n \in \Z_+}}$ the formal matrix representation of $T_\vp$ with respect to the basis $\{f_n\}_{n \in \Z_+} \in B_{H^2(\D)}$. Observe that if $T_{\vp}$ is $S$-Toeplitz, then there exists a sequence $\{\alpha_n\}_{n\in \Z_+}$ such that
\[
T_\vp f_k = \sum_{n=-k}^{\infty} \alpha_n f_{n+k} \qquad (k \in \Z_+),
\]
which yields
\[
[T_\vp]_{\{f_n\}_{n \in \Z_+}}^t =
\begin{bmatrix}	\alpha_0 & \alpha_1 & \alpha_2 & \alpha_3 &\hdots
\\
\alpha_{-1} & \alpha_0 & \alpha_1 & \alpha_2 & \ddots
\\
\alpha_{-2} & \alpha_{-1} & \alpha_0 & \alpha_1 & \ddots
\\
\alpha_{-3} & \alpha_{-2} & \alpha_{-1} & \alpha_0 & \ddots
\\
\vdots&\vdots&\vdots&\vdots&\ddots
\end{bmatrix}.
\]
Now we turn to $C$-symmetric Toeplitz operators. Let $C \in \clc(H^2(\D))$. In view of Proposition \ref{Prop: C = UC =CU*} (also see Definition \ref{def: con factor}), $C$ admits the canonical factorization, that is, there is a unique unitary $U \in \clb(H^2(\D))$ such that
\[
C = U J_{H^2(\D)} = J_{H^2(\D)} U^*.
\]
Moreover, by Lemma \ref{lemma: C = UC =CU*}, we know that
\[
S := CM_zC,
\]
is a shift corresponding to $\{f_n\}_{n \in \Z_+} \in B_{H^2(\D)}$, where $f_n:= Uz^n$, $n \in \Z_+$. We now connect the formal Toeplitz matrix $[T_\vp]_{\{f_n\}_{n \in \Z_+}}$ with symmetricity of $T_\vp$.

\begin{theorem}\label{thm: S toep 1}
Let $\vp \in L^\infty(\T)$, and let $C \in \clc(H^2(\D))$. Suppose $C = U J_{H^2(\D)}$ is the canonical factorization of $C$, and let $f_n:= Uz^n$, $n \in \Z_+$. Then $T_\vp$ is $C$-symmetric if and only if
\[
[T_\vp]_{\{f_n\}_{n \in \Z_+}} = [T_\vp]_{\{z^n\}_{n \in \Z_+}}^t.
\]
\end{theorem}
\begin{proof}
By the definition of symmetric operators, $T_\vp$ is $C$-symmetric if and only if $C T_\vp z^k = T_\vp^* C z^k$ for all $k \in \Z_+$. In view of the canonical factorization
\[
C = U J_{H^2(\D)},
\]
for each $k \in \Z_+$, we compute
\[
\begin{split}
C T_\vp z^k & = C P_{H^2(\D)} (\sum_{n=-\infty}^{\infty} \vp_n z^{n+k})
\\
& = U J_{H^2(\D)} (\sum_{n=-k}^{\infty} \bar{\vp}_n z^{n+k})
\\
& = U (\sum_{n=-k}^{\infty} \bar{\vp}_n z^{n+k})
\\
& = \sum_{n=-k}^{\infty} \bar{\vp}_n f_{n+k},
\end{split}
\]
as $U z^n = f_n$ for all $n \in \Z_+$. On the other hand, since
\[
U J_{H^2(\D)} z^k = U z^k = f_k,
\]
we have
\[
T_\vp^* C z^k = T_\vp^* f_k.
\]
This implies that $T_\vp$ is $C$-symmetric with $C = U J_{H^2(\D)}$ if and only if
\[
T_\vp^* f_k = \sum_{n=-k}^{\infty} \bar{\vp}_n f_{n+k} \qquad (k \in \Z_+).
\]
Equivalently, we have
\[
[T_\vp]_{\{f_n\}_{n \in \Z_+}} =
\begin{bmatrix}	\vp_0 & \vp_1 & \vp_2 & \vp_3 &\hdots
\\
\vp_{-1} & \vp_0 & \vp_1 & \vp_2 & \ddots
\\
\vp_{-2} & \vp_{-1} & \vp_0 & \vp_1 & \ddots
\\
\vp_{-3} & \vp_{-2} & \vp_{-1} & \vp_0 & \ddots
\\
\vdots&\vdots&\vdots&\vdots&\ddots
\end{bmatrix}.
\]
Therefore, $T_\vp$ is $C$-symmetric with $C = U J_{H^2(\D)}$ if and only if
\[
[T_\vp]_{\{f_n\}_{n \in \Z_+}} = [T_\vp]_{\{z^n\}_{n \in \Z_+}}^t,
\]
which completes the proof of the theorem.
\end{proof}

Recall Lemma \ref{lemma: C = UC =CU*}: If $C$ is a conjugation on $H^2(\D)$ with the canonical factorization $C = U C_{H^2(\D)}$ for some unitary $U \in \clb(H^2(\D))$, then
\[
S:= C M_zC,
\]
is the shift corresponding to $\{f_n\}_{n\in \Z_+} \in B_{H^2(\D)}$, where $f_n:= U z^n$, $n \in \Z_+$. Suppose $\vp \in L^\infty(\T)$. Evidently, if
\[
[T_\vp]_{\{f_n\}_{n \in \Z_+}} = [T_\vp]_{\{z^n\}_{n \in \Z_+}}^t,
\]
then, in particular, $T_{\vp}$ is $S$-Toeplitz. This was already observed in Proposition \ref{thm: symm S Toep}. Here the above equality of formal Toeplitz matrices is the missing link in the classification of symmetric Toeplitz operators (see the paragraph following Proposition \ref{thm: symm S Toep}).

We now consider a variation of the above argument.

\begin{theorem}\label{thm: S toep 2}
Let $\vp \in L^\infty(\T)$, and let $C \in \clc(H^2(\D))$. Then $T_\vp$ is $C$-symmetric if and only if
\[
[T_\vp]_{\{z^n\}_{n \in \Z_+}}^t = [U^* T_{\vp} U]_{\{z^n\}_{n \in \Z_+}},
\]
where $C = U J_{H^2(\D)}$ is the canonical factorization of $C$.
\end{theorem}
\begin{proof}
Let $C$ be a conjugation on $H^2(\D)$, and let $T_\vp$, $\vp \in L^\infty(\T)$, be a Toeplitz operator. By the canonical factorization of $C$, there is a unique unitary $U \in \clb(H^2(\D))$ such that
\[
C = U J_{H^2(\D)} = J_{H^2(\D)} U^*.
\]
Therefore
\[
C T_\vp C = U (J_{H^2(\D)} T_\vp J_{H^2(\D)}) U^*,
\]
and hence, $T_\vp$ is $C$-symmetric if and only if
\begin{equation}\label{eqn: Ustar T U}
U^* T_\vp U = J_{H^2(\D)} T_\vp^* J_{H^2(\D)}.
\end{equation}
For each $m, n \in \Z_+$, \eqref{eqn: Toepl coeff} and the above equality imply
\[
\begin{split}
\vp_{m-n} & = \langle T_\vp z^n, z^m \rangle
\\
& =  \langle z^n, T_\vp ^* z^m \rangle
\\
& =  \langle J_{H^2(\D)} T_\vp ^* z^m, z^n \rangle
\\
& =  \langle J_{H^2(\D)} T_\vp ^* J_{H^2(\D)} z^m, z^n \rangle
\\
& =  \langle (U^* T_\vp U) z^m, z^n \rangle.
\end{split}
\]
Consequently, $T_\vp$ is $C$-symmetric if and only if $[T_\vp]_{\{z^n\}_{n \in \Z_+}}^t = [U^* T_{\vp} U]_{\{z^n\}_{n \in \Z_+}}$. This completes the proof of the theorem.
\end{proof}

In particular, if $T_\vp$ is $C$-symmetric, then $U^* T_{\vp} U$ is also a Toeplitz operator. Moreover, a closer inspection reveals that the same proof of the above theorem yields:

\begin{theorem}\label{thm: charac symm op}
Let $\clh$ be a Hilbert space, $\{e_n\}_{n \in \Lambda} \in B_\clh$ be the canonical basis, and let $C\in \clc(\clh)$. Then $T\in \mathcal{B}(\clh)$ is $C$-symmetric if and only if
\[
[T]_{\{e_n\}_{n \in \Lambda}}^t = [U^*TU]_{\{e_n\}_{n \in \Lambda}},
\]
where $C=UJ_\clh$ is the canonical factorization of the conjugation $C$.
\end{theorem}

It is worthwhile to note that the canonical basis of $\clh$ can be chosen as per convenience or requirement. Therefore, the above observation is different from \eqref{eqn: classi: old symm}.

Now we proceed in unfolding the matrix equality of Theorem \ref{thm: S toep 2}. For each $n \in \Z_+$, we write
\[
U z^n = \sum_{j=0}^{\infty} u_{n,j} z^j,
\]
where $u_{n,j} = \langle U z^n, z^j \rangle$, $j \in \Z_+$. Since $Uz^n = C z^n$, we have
\[
u_{n,j} = \langle U z^n, z^j \rangle = \langle C z^n, z^j \rangle \qquad (n, j \in \Z_+).
\]
By Proposition \ref{prop: U is symmetric}, it follows that
\begin{equation}\label{eqn: u twist}
u_{m,j} = u_{j,m} \qquad (m, j \in \Z_+).
\end{equation}
Then, by \eqref{eqn: Toepl coeff} and Theorem \ref{thm: S toep 2}, $T_\vp$ is $C$-symmetric if and only if
\[
\vp_{m-n} = \langle U^* T_{\vp} U z^m, z^n \rangle \qquad (m, n \in \Z_+).
\]
For each $m, n \in \Z_+$, we compute
\[
\begin{split}
\langle U^* T_{\vp} U z^m, z^n \rangle & = \langle T_\vp U z^m, U z^n \rangle
\\
& = \sum_{i,j = 0}^{\infty} u_{m,j} \overline{u_{n,i}} \langle T_\vp z^j, z^i \rangle
\\
& = \sum_{i,j = 0}^{\infty} u_{m,j} \overline{u_{n,i}} \vp_{i-j}
\\
& = \sum_{i,j = 0}^{\infty} u_{m,j} \vp_{i-j} \overline{u_{i,n}}.
\end{split}
\]
Thus we have proved:

\begin{corollary}\label{cor: phi and u}
Let $\vp = \sum_{n= - \infty}^{\infty} \vp_n z^n \in L^{\infty}(\T)$, and let $C$ be a conjugation with canonical factorization $C = UJ_{H^2(\D)}$. The following are equivalent:
\begin{enumerate}
\item $T_{\vp}$ is $C$-symmetric.

\item $\vp_{m-n} = \langle U^* T_{\vp} U z^m, z^n \rangle$ for all $m, n \in \Z_+$.

\item $\vp_{m-n} = \sum_{i,j = 0}^{\infty} u_{m,j} \vp_{i-j} \overline{u_{i,n}}$ for all $m, n \in \Z_+$, where
\[
u_{i,j} := \langle U z^i, z^j \rangle = \langle C z^i, z^j \rangle \qquad (i, j \in \Z_+).
\]
\end{enumerate}
\end{corollary}

We now illustrate Corollary \ref{cor: phi and u} in the setting of Example \ref{illustration_example}. Fix $\theta, \xi \in \R$. Recall that
\[
(C_{\theta, \xi} f)(z)=e^{i\xi}\overline{f(e^{i\theta}\bar{z})} \qquad (f \in H^2(\D)),
\]
defines a conjugation on $H^2(\D)$ with respect to $\{e^{\frac{i\xi}{2}} e^{\frac{-i n \theta}{2}}z^n\} \in B_{H^2(\D)}$. The following was first proved by Ko and Lee \cite{Ko_Lee_SCTO}.

\begin{corollary}\label{cor: Ko Toeplitz}
Let $\varphi(z)= \sum_{n=-\infty}^{\infty}{\varphi}_n z^n \in L^\infty$. Then $T_\varphi$ is $C_{\xi,\theta}$-symmetric if and only if
\[
\vp_n = \varphi_{-n} e^{-in\theta} \qquad (n \in \Z).
\]
\end{corollary}
\begin{proof}
For each $n \in \Z_+$, note that
\[
U z^n = C_{\theta, \xi} J_{H^2(\D)} z^n = e^{i\xi} e^{-i n\theta}z^n.
\]
Hence
\[
\begin{split}
\inp{U^*T_{\vp}U z^n}{z^m} & = \inp{T_{\vp}U z^n}{Uz^m}
\\
& = \inp{T_{\vp}(e^{i\xi}e^{-in\theta}z^n)}{e^{i\xi}e^{-im\theta}z^m}
\\
& = e^{i(m-n)\theta}\inp{T_{\vp}z^{n}}{z^m}.
\end{split}
\]
Then, by Corollary \ref{cor: phi and u}, $T_{\vp}$ is $C_{\theta,\xi}$-symmetric if and only if
\beqn
e^{i(m-n)\theta} {\vp}_{m-n} = {\vp}_{n-m},
\eeqn
for all $m,n \in \Z_+$, or equivalently
\beqn
e^{i n\theta} {\vp}_{n} = {\vp}_{-n} \qquad (n \in \Z_+).
\eeqn
\end{proof}

Observe that the matrix representation $[T_\vp]$ of the $C_{\xi,\theta}$-symmetric Toeplitz operator $T_\varphi$ with respect to the basis $\{e^{\frac{i\xi}{2}} e^{\frac{-i n \theta}{2}}z^n\}_{n \in \Z_+} \in B_{H^2(\D)}$ is given by
\[
[T_\vp] = \begin{bmatrix}	\vp_0 & e^{\frac{i\theta}{2}} \vp_1 & e^{\frac{2i\theta}{2}} \vp_2 & e^{\frac{3i\theta}{2}} \vp_3 &\hdots
\\
e^{\frac{i\theta}{2}} \vp_1 & \vp_0 & e^{\frac{i\theta}{2}} \vp_2 & e^{\frac{2i\theta}{2}} \vp_1 & \ddots
\\
e^{\frac{2i\theta}{2}} \vp_2 & e^{\frac{i\theta}{2}} \vp_1 & \vp_0 &
e^{\frac{i\theta}{2}} \vp_1 & \ddots
\\
e^{\frac{3i\theta}{2}} \vp_3 & e^{\frac{2i\theta}{2}} \vp_2 & e^{\frac{i\theta}{2}} \vp_1 & \vp_0 & \ddots
\\
\vdots&\vdots&\vdots&\vdots&\ddots
\end{bmatrix}.
\]
Evidently, as also follows from the fact that $T_\vp$ is $C_{\xi,\theta}$-symmetric, we have that
\[
[T_\vp]^t = [T_\varphi].
\]

\begin{remark}
The above matrix representation was observed in \cite[p~26]{Ko_Lee_SCTO} for a special class of $L^\infty(\T)$ symbols.
\end{remark}

Now we turn to Example \ref{example Ferreira}. Recall that, for a fixed sequence $\{\alpha_n\}_{n \in \Z_+} \subseteq \T$, $C_{\alpha}$ is a conjugation on $H^2(\D)$, where
\[
C_{\alpha}(\sum_{n=0}^\infty a_nz^n) = \sum_{n=0}^\infty \bar{a}_n \alpha_n^{2} z^n,
\]
for all $\sum_{n=0}^\infty a_nz^n \in H^2(\D)$. Again
\[
U z^n = (C_{\alpha}J_{H^2(\D)}) z^n = C_\alpha z^n = \alpha_n^{2} z^n \qquad (n \in \Z_+).
\]
This implies
\[
\begin{split}
\inp{U^*T_{\vp}U(z^m)}{z^n} & =  \inp{T_{\vp} \alpha_n^{2} z^m}{\alpha_m^{2}z^n}
\\
& = \alpha_{m}^{2}\bar{\alpha}_n^{2}\inp{T_{\vp}z^{m}}{z^n}
\\
& = \alpha_{m}^{2}\bar{\alpha}_n^{2} \vp_{n-m},
\end{split}
\]
and hence, by Corollary \ref{cor: phi and u}, $T_\vp$ is $C_\alpha$-symmetric if and only if
\begin{equation}\label{eq: example C alpha}
\vp_{m-n} = \alpha_{m}^{2}\bar{\alpha}_n^{2} \vp_{n-m} \qquad (m, n \in \Z_+).
\end{equation}
On the other hand
\[
\begin{split}
u_{i,j} & = \langle U z^i, z^j \rangle
\\
& = \begin{cases}
\alpha_i^{2} & \mbox{if } i = j
\\
0 & \mbox{otherwise},
\end{cases}
\end{split}
\]
and hence, in this case, for each fixed $m, n \in \Z_+$, we have
\[
\sum_{i,j=0}^{\infty} u_{m,j} \vp_{i-j} \overline{u_{i,n}} = \alpha_{m}^{2}\bar{\alpha}_n^{2} \vp_{n-m},
\]
which again, by Corollary \ref{cor: phi and u}, verifies that $T_\vp$ is $C_\alpha$-symmetric if and only if \eqref{eq: example C alpha} holds. Therefore, we have:

\begin{corollary}\label{example Junior}
Let $\varphi(z)= \sum_{n=-\infty}^{\infty}{\varphi}_n z^n \in L^\infty(\T)$. Then $T_\varphi$ is $C_{\alpha}$-symmetric if and only if
\[
{\varphi}_{m-n} = \alpha_m^2 \bar{\alpha}_n^2 {\varphi}_{n-m} \qquad (m, n \in \Z_+).
\]
\end{corollary}

This was observed in \cite[Theorem 9]{ferreira_junior}. However, along with new proof, the present version fixes an error in the statement \cite[Theorem 9]{ferreira_junior}.

\newsection{Yet another characterization}\label{sec: Conj and Toeplitz}

We continue to follow the analysis of Section \ref{sec: rep of conj} and present our final characterization of symmetric Toeplitz operators. The key is the representations of conjugations on $H^2(\D)$ with respect to the canonical basis
\[
\zeta := \{z^n\}_{n \in \Z_+} \in B_{H^2(\D)}.
\]
Following our usual convention, given $\{f_n\}_{n \in \Z_+} \in B_{H^2(\D)}$, we write
\[
c_{n,m}^{(\zeta)}=\displaystyle\sum_{k=0}^\infty \langle f_k, z^n \rangle \langle f_k, z^m \rangle \qquad (m, n \in \Z_+).
\]
With this notation at hand we can now state the main result of this section.

\begin{theorem}\label{thm_main}
Let $\varphi(z) = \sum_{n=-\infty}^{\infty} {\varphi}_{n} z^n \in L^\infty(\T)$, and let $C$ be a conjugation on $H^2(\D)$ corresponding to $\{f_n\}_{n \in \Z_+} \in B_{H^2(\D)}$. Then $T_\varphi$ is $C$-symmetric if and only if
\[
\sum_{n=0}^{\infty}\overline{{\varphi}_{n-k}}c_{n,j}^{(\zeta)} = \sum_{n=1}^{\infty}\overline{{\varphi}_n}c_{k, n+j}^{(\zeta)} + \sum_{l=0}^{j}\overline{{\varphi}_{-l}}c_{k,j-l}^{(\zeta)},
\]
for all $j,k\in \Z_+$.
\end{theorem}

\begin{proof}
By Proposition \ref{conju_Charac_theorem}, we know that
\[
C(\sum_{n=0}^\infty a_n z^n) = \sum_{n=0}^\infty \sum_{m=0}^\infty \bar{a}_n c_{n, m}^{(\zeta)} z^m,
\]
for all $\{a_n\} \in \ell^2$, where  $\zeta = \{z^n\}_{n \geq 0} \in B_{H^2(\D)}$ and $c_{n,m}^{(\zeta)}$ are defined as above. Note that $CT_\varphi C = T_{\vp}^*$ if and only if $C T_\varphi = T_{\vp}^* C$, which is equivalent to the condition that
\[
C T_\varphi z^k = T_{\vp}^* C z^k \qquad (k \in \Z_+).
\]
Fix $k \in \Z_+$. We compute
\[
\begin{split}
CT_\varphi z^k & = CP_{H^2(\mathbb{D})} (\sum_{n=-\infty}^{\infty}  {\varphi}_n z^{n+k})
\\
& = C (\sum_{n=-k}^{\infty} {\varphi}_n z^{n+k})
\\
& = \sum_{n=-k}^{\infty} \sum_{m=0}^{\infty}\overline{{\varphi}_n}c_{n+k, m}^{(\zeta)}z^m
\\
& = \sum_{j=0}^{\infty} (\sum_{n=0}^{\infty}\overline{{\varphi}_{n-k}}c_{n, j}^{(\zeta)})z^j.
\end{split}
\]
On the other hand
\[
\begin{split}
T^*_\varphi Cz^k& = T^*_\varphi (\sum_{m=0}^{\infty}c_{k, m}^{(\zeta)}z^m)
\\
& = P_{H^2(\mathbb{D})} (\sum_{n=-\infty}^{\infty}\overline{{\varphi}_n}z^{-n}\sum_{m=0}^{\infty}c_{k, m}^{(\zeta)}z^m)
\\
& = P_{H^2(\mathbb{D})} (\sum_{n=-\infty}^{\infty}\overline{{\varphi}_{-n}}z^{n}\sum_{m=0}^{\infty}c_{k, m}^{(\zeta)}z^m)
\\
& = \sum_{n=1}^{\infty}\overline{{\varphi}_n}\sum_{m=n}^{\infty}c_{k, m}^{(\zeta)}z^{m-n} + \sum_{n=0}^{\infty}\overline{{\varphi}_{-n}}\sum_{m=0}^{\infty}c_{k, m}^{(\zeta)}z^{m+n}
\\
& = \sum_{j=0}^{\infty} (\sum_{n=1}^{\infty}\overline{{\varphi}_n}c_{k, n+j}^{(\zeta)}+\sum_{l=0}^{j}\overline{{\varphi}_{-l}}c_{k, j-l}^{(\zeta)})z^j.
\end{split}
\]
Therefore, $C T_\varphi z^k = T_{\vp}^* C z^k$ if and only if
\[
\sum_{j=0}^{\infty} (\sum_{n=0}^{\infty}\overline{{\varphi}_{n-k}} c_{n, j}^{(\zeta)})z^j = \sum_{j=0}^{\infty} (\sum_{n=1}^{\infty}\overline{{\varphi}_n} c_{k, n+j}^{(\zeta)}+\sum_{l=0}^{j}\overline{{\varphi}_{-l}}c_{k, j-l}^{(\zeta)})z^j.
\]
By equating the coefficient of $z^j$ on either side of the above equality, we find that $C T_\varphi = T_{\vp}^* C$ if and only if
\[
\sum_{n=0}^{\infty}\overline{{\varphi}_{n-k}} c_{n, j}^{(\zeta)} =  \sum_{n=1}^{\infty}\overline{{\varphi}_n} c_{k, n+j}^{(\zeta)}+\sum_{l=0}^{j}\overline{{\varphi}_{-l}} c_{k, j-l}^{(\zeta)},
\]
for all $j,k\in \Z_+$, which completes the proof of the theorem.
\end{proof}

Like all the previous characterizations, Theorem \ref{thm_main} also unifies all the existing results on classifications of specific classes of conjugate Toeplitz operators corresponding to specific classes of conjugations. On the other hand, the above result transfers the problem of characterization of symmetric Toeplitz operators to infinitely many equations, which also explains the challenges to the classification problem and in all the partial findings in the literature (cf. \cite{Guo, Ko_Lee_SCTO}).

One can summarize the infinitely many conditions in the conclusion of Theorem \ref{thm_main} in a simpler form. First, we rewrite the equality
\[
\sum_{n=0}^{\infty}\overline{{\varphi}_{n-k}}c_{n, j}^{(\zeta)}- \sum_{n=1}^{\infty}\overline{{\varphi}_n}c_{k, n+j}^{(\zeta)}= \sum_{l=0}^{j}\overline{{\varphi}_{-l}}c_{k, j-l}^{(\zeta)},
\]
as
\[
\sum_{p=0}^{k}\overline{{\varphi}_{-p}} c_{k-p, j}^{(\zeta)}+\sum_{n=1}^{\infty} (c_{n+k, j}^{(\zeta)}-c_{k, n+j}^{(\zeta)}) \overline{{\varphi}_n} = \sum_{l=0}^{j}\overline{{\varphi}_{-l}}c_{k, j-l}^{(\zeta)},
\]
which implies
\[
\sum_{n=1}^{\infty} (c_{n+k, j}^{(\zeta)}-c_{k, n+j}^{(\zeta)}) \overline{{\varphi}_n} = \sum_{l=0}^{j}\overline{{\varphi}_{-l}} c_{k, j-l}^{(\zeta)}-\sum_{p=0}^{k}\overline{{\varphi}_{-p}} c_{k-p, j}^{(\zeta)},
\]
for all $j,k\in \Z_+$. In particular, $j=k$ yields
\[
\sum_{n=1}^{\infty}(c_{n+k, k}^{(\zeta)} - c_{k, n+k}^{(\zeta)}) \overline{{\varphi}_n} = \sum_{p=0}^{k}(c_{k, k-p}^{(\zeta)} - c_{k-p, k}^{(\zeta)})\overline{{\varphi}_{-p}}.
\]
And, in general, if we set
\[
\Phi_+ = (\overline{{\varphi}_1},\overline{{\varphi}_2} ,\overline{{\varphi}_3} , \ldots ),
\]
and
\[
\Phi_- = (\overline{{\varphi}_{-1}}, \overline{{\varphi}_{-2}}, \overline{{\varphi}_{-3}},\ldots ),
\]
then the above set of equalities can be expressed in the following formal matrix equation
\begin{equation}\label{eqn: matrix repr Phi}
X(k)\Phi_+=Y(k)\Phi_- \qquad (k \in \Z_+),
\end{equation}
where
\[
X(k)=\begin{bmatrix}
(c_{1+k,0}^{(\zeta)}-c_{k, 1}^{(\zeta)}) & (c_{2+k, 0}^{(\zeta)}-c_{k,2}^{(\zeta)}) & (c_{3+k, 0}^{(\zeta)}-c_{k,3}^{(\zeta)}) & \ddots
\\
(c_{1+k, 1}^{(\zeta)} - c_{k, 2}^{(\zeta)}) & (c_{2+k, 1}^{(\zeta)}-c_{k, 3}^{(\zeta)}) & (c_{3+k,1}^{(\zeta)} - c_{k, 4}^{(\zeta)}) & \ddots
\\
(c_{1+k,2}^{(\zeta)} - c_{k,3}^{(\zeta)}) & (c_{2+k,2}^{(\zeta)} - c_{k,4}^{(\zeta)}) & (c_{3+k,2}^{(\zeta)} - c_{k, 5}^{(\zeta)})& \ddots
\\
\vdots&\vdots&\vdots& \ddots
\\
(c_{1+k, k}^{(\zeta)}-c_{k,1+k}^{(\zeta)}) & (c_{2+k,k}^{(\zeta)} - c_{k, 2+k}^{(\zeta)}) & (c_{3+k, k}^{(\zeta)}-c_{k, 3+k}^{(\zeta)}) & \ddots
\\
\vdots&\vdots&\vdots& \ddots
\end{bmatrix},
\]
and
\[
Y(k)=\begin{bmatrix} -c_{k-1,0}^{(\zeta)} & -c_{k-2,0}^{(\zeta)} & -c_{k-3,0}^{(\zeta)} & \cdots & -c_{0,0}^{(\zeta)}&0&0&\ddots
\\
(c_{k,0}^{(\zeta)} - c_{k-1,1}^{(\zeta)}) & - c_{k-2,1}^{(\zeta)} & -c_{k-3,1}^{(\zeta)} & \cdots & -c_{0,1}^{(\zeta)}&0&0&\ddots
\\
(c_{k,1}^{(\zeta)} - c_{k-1,2}^{(\zeta)}) & (c_{k,0}^{(\zeta)} - c_{k-2, 2}^{(\zeta)}) & -c_{k-3, 2}^{(\zeta)} & \cdots & -c_{0, 2}^{(\zeta)}&0&0&\ddots
\\
\vdots&\vdots&\vdots&\vdots&\vdots&\vdots&\vdots& \ddots
 \\
(c_{k, k-1}^{(\zeta)} - c_{k-1,k}^{(\zeta)}) & (c_{k, k-2}^{(\zeta)} - c_{k-2, k}^{(\zeta)}) & (c_{k, k-3}^{(\zeta)}-c_{k-3,k}^{(\zeta)})&\cdots & (c_{k, 0}^{(\zeta)} - c_{0,k}^{(\zeta)}) & 0& 0 &\ddots
\\
\vdots&\vdots&\vdots&\vdots&\vdots&\vdots&\vdots & \ddots
\end{bmatrix},
\]
for all $k \in \Z_+$. We summarize this as:

\begin{corollary}
A Toeplitz operator $T_\varphi$ with symbol $\varphi= \sum_{n=-\infty}^{\infty} {\varphi}_n z^n \in L^\infty(\T)$ is $C$-symmetric if and only if
\[
X(n)\Phi_+=Y(n)\Phi_- \qquad (n \in \Z_+).
\]
\end{corollary}

We wish to point out that the results of this section do not depend on the canonical factorizations of conjugations. The implication of the results of this section for (finite) Toeplitz matrices will be explained in Theorem \ref{finiteToeplitzCS}.

\newsection{Toeplitz operators on the polydisc}\label{sec: polydisc}

In this section, we characterize complex symmetric Toeplitz operators on the Hardy space over the unit polydisc $\D^d$, where $\D^d = \{\z = (z_1, \ldots, z_d) \in \C^d: |z_j| < 1, j=1, \ldots, d\}$. We consider from now on $d$ a natural number such that $d \geq 2$. We carefully adopt the notion of $S$-Toeplitz operators in several variables and extend the classification results of symmetric Toeplitz operators of Section \ref{sec: S Toeplitz and symm} to $\D^d $.

Recall that $H^2(\D^d)$, the Hardy space over $\D^d$, is the Hilbert space
of all analytic functions $f = \sum_{\bk \in \Z_+^d} a_{\bk} z^{\bk}$ on $\D^d$ such that (cf. \cite{Rudin})
\[
\|f\|_{H^2(\mathbb{D}^d)} = (\sum_{\bk \in \Z_+^d} |a_{\bk}|^2)^{\frac{1}{2}} < \infty.
\]
Here, $\bk = (k_1, \ldots, k_d) \in \Z_+^d$ and $z^{\bk} = z_1^{k_1} \cdots z_d^{k_d}$. As in the case of one variable, we often identify $H^2(\mathbb{D}^d)$ (via radial limits of square summable analytic functions) with the closed subspace of all functions $f = \sum_{\bk \in \Z^d} a_{\bk} z^{\bk} \in L^2(\T^d)$ such that $a_{\bk} = 0$ whenever $k_j < 0$ for some $j=1, \ldots, d$. Also recall that
\[
\{z^{\bk}\}_{\bk \in \Z_+^d} \in B_{H^2(\D^d)} \text{ and } \{z^{\bk}\}_{\bk \in \Z^d} \in B_{L^2(\T^d)}.
\]
Denote by $P_{H^2(\mathbb{D}^n)}$ the orthogonal projection from $L^2(\mathbb{T}^d)$ onto $H^2(\mathbb{D}^d)$, that is
\[
P_{H^2(\mathbb{D}^d)} (\sum_{\bk \in \Z^d} a_{\bk} z^{\bk}) = \sum_{\bk \in \Z_+^d} a_{\bk} z^{\bk},
\]
for all $\sum_{\bk \in \Z^d} a_{\bk} z^{\bk} \in L^2(\T^d)$. As in the case of one variable, the Toeplitz operator $T_{\varphi}$ on $H^2(\mathbb{D}^d)$ with symbol $\varphi \in L^\infty(\T^d)$ is defined by
\[
T_{\varphi} = P_{H^2(\mathbb{D}^d)} L_{\varphi}|_{H^2(\D^d)},
\]
where $L_{\varphi}$ is the Laurent operator on $L^2(\mathbb{T}^d)$. In other words, $T_{\varphi} f = P_{H^2(\mathbb{D}^d)} ({\varphi} f)$ for all $f \in H^2(\mathbb{D}^d)$. Recall \cite[Theorem 3.1]{maji_sarkar_srijan} that an operator $T \in \clb(H^2(\D^n))$ is Toeplitz if and only if
\[
M^*_{z_{i}} T M_{z_i} = T \qquad (i = 1, \ldots, d).
\]
Note that the $(M_{z_1}, \ldots, M_{z_d})$ is a $d$-tuple of commuting shifts on $H^2(\D^d)$, where
\[
M_{z_i} f = z_i f \qquad (f \in H^2(\D^d)),
\]
and $i=1, \ldots, d$. For more on Toeplitz operators on the polydisc, we refer the reader to \cite{Pradhan et al, maji_sarkar_srijan} and the reference therein. Before going further, we need a lemma in the line of Corollary \ref{cor: CMzC shift}.

Before going further, we record a simple observation: Let $X \in \clb(\clh)$, and let $C \in \clc(\clh)$. Then
\[
(CXC)^{*n} = C X^{*n} C \qquad (n \in \Z_+).
\]
Indeed, for each $f, g \in \clh$, applying \eqref{eqn: conj inner prod} repeatedly, we find
\[
\begin{split}
\langle (CXC)^* f, g \rangle & = \langle f, (CXC) g \rangle
\\
& = \langle XC g, Cf \rangle
\\
& = \langle C g, X^* C f \rangle
\\
& = \langle CX^*  C f, g \rangle.
\end{split}
\]
This implies that $(CXC)^* = C X^*C$, and then, by induction, we conclude that
\begin{equation}\label{eqn: CXC star n}
(CXC)^{*n} = C X^{*n} C \qquad (n \in \Z_+).
\end{equation}
We are now ready for a general version of Lemma \ref{lemma: C = UC =CU*}. Recall that the multiplicity of a shift $S$ is the number (see the first paragraph of Section \ref{sec: S Toeplitz})
\[
\text{mult} S = \text{dim} (\ker S^*) \in \N \cup \{\infty\}.
\]

\begin{lemma}\label{lemma: conju_shift-new}
Let $C \in \clc(\clh)$, and let $S \in \clb(\clh)$ be a shift. Then $CSC$ is a shift, and
\[
\text{mult} (CSC) = \text{mult} S.
\]
\end{lemma}
\begin{proof}
We already know that $(CSC)^{*n} = C S^{*n} C$ for all $n \in \Z_+$. For each $f \in \clh$, we have
\[
\|(CSC)^{*n} f \| = \|(CS^{*n}C) f \| = \|S^{*n} (Cf)\| \longrightarrow 0,
\]
as $n \raro \infty$. This and
\[
\|(CSC)f\| = \|SCf\| = \|Cf\| = \|f\|,
\]
imply that $CSC$ is a shift. Now we set
\[
\clw := C(\ker S^*).
\]
Since $C$ is a conjugation, it follows that $C(\clw) = \ker S^*$, and
\[
\text{dim} \clw = \text{dim} (\ker S^*).
\]
In particular, if $f = Cg$ for some $g \in \ker S^*$, then $Cf = g$, and hence
\[
(CS^*C) f = CS^* g = 0,
\]
which implies that $\clw \subseteq \ker(CS^*C)$. On the other hand, if $f\in \ker (CS^*C)$, then
\[
S^*Cf = 0,
\]
that is, $Cf \in \ker S^*$. Since $\clw = C(\ker S^*)$, we conclude that
\[
f = C^2f \in \clw.
\]
Therefore, $\ker (CS^*C) \subseteq \clw$, which implies that $\ker (CS^*C) = \clw$. In particular
\[
\text{dim} \clw = \text{dim} (\ker S^*) = \text{dim} (\ker (CS^*C)).
\]
This completes the proof of the lemma.
\end{proof}

A $d$-tuple of commuting isometries $\bm{S} = (S_1, \ldots, S_d)$ on a Hilbert space $\clh$ is said to be \textit{doubly commuting shift} if $S_t$ is a shift, $t = 1, \ldots, d$, and
\[
S_i^* S_j = S_j S_i^* \qquad (i \neq j).
\]
The multiplicity of a doubly commuting shift $\bm{S} = (S_1, \ldots, S_d)$ is defined by
\[
\text{mult} \bm{S} = \text{dim }\Big[\bigcap_{i=1}^d \ker S_i^*\Big].
\]
Of course, $(M_{z_1}, \ldots, M_{z_d})$ on $H^2(\D^d)$ is a doubly commuting shift of multiplicity one (cf. \cite{JS-14}). We now present yet another context in which doubly commuting shift appears naturally: Let $C \in \clc(H^2(\D^d))$, and suppose
\[
S_i = C M_{z_i} C \qquad (i=1, \ldots, d).
\]
Lemma \ref{lemma: conju_shift-new} implies that $S_i$ is a shift, $i=1, \ldots, d$. Moreover, by \eqref{eqn: CXC star n}, for each $i \neq j$, we have
\[
\begin{split}
S_{i} S^*_{j} & = (CM_{z_i}C)(CM^*_{z_j}C)
\\
& = CM_{z_i}M^*_{z_j}C
\\
& =  CM^*_{z_j}M_{z_i}C
\\
& = CM^*_{z_j} C CM_{z_i}C
\\
& = S^*_{j} S_{i}.
\end{split}
\]
Finally, since (see \eqref{eqn: CXC star n} again)
\[
\ker (C M_{z_i} C)^* = \ker (C M_{z_i}^* C) = C (\ker M_{z_i}^*),
\]
for all $i=1, \ldots, d$, it follows that
\[
\begin{split}
\bigcap_{i=1}^d \ker S_i^* & = \bigcap_{i=1}^d (C (\ker M_{z_i}^*))
\\
& = C \Big(\bigcap_{i=1}^d (\ker M_{z_i}^*)\Big).
\end{split}
\]
As we mentioned earlier, the multiplicity of $(M_{z_1}, \ldots, M_{z_d})$ on $H^2(\D^d)$ is one. Consequently
\[
\text{mult}(S_1, \ldots, S_d) = 1,
\]
and thus we have proved:

\begin{lemma}\label{lemma: Conj shift d var}
Let $C\in \clb_{a}(H^2(\D^d))$, and suppose $S_i = C M_{z_i} C$ for all $i=1, \ldots, d$. Then $(S_1, \ldots, S_d)$ on $H^2(\D^d)$ is a $d$-tuple of doubly commuting shift of multiplicity one.
\end{lemma}

This sets the stage for the notion of $\bm{S}$-Toeplitz operators in several variables.

\begin{definition}
Let $\bm{S} = (S_1, \ldots, S_d)$ be a $d$-tuple of doubly commuting shift on $\clh$. An operator $T \in  \clb(\clh)$ is said
to be $\bm{S}$-Toeplitz if
\[
S_i^* T S_i = T \qquad (i = 1, \ldots, d).
\]
\end{definition}

Therefore, $T \in \clb(\clh)$ is $\bm{S}$-Toeplitz if and only if $T$ is $S_i$-Toeplitz for all $i=1, \ldots, d$. In view of Lemma \ref{lemma: Conj shift d var} above, the proof of the following proposition is now essentially the same as that of Proposition \ref{thm: symm S Toep}.

\begin{proposition}
Let $\varphi\in L^\infty(\T^{d})$, and let $C \in \clc(H^2(\D^{d}))$. If $T_\vp$ is $C$-symmetric, then $T_{\vp}$ is $\bm{S}$-Toeplitz, where $\bm{S} = (CM_{z_1} C, \ldots, C M_{z_d} C)$.
\end{proposition}

Let $C \in \clc(H^2(\D^d))$. Suppose
\[
C = U J_{H^2(\D^d)},
\]
the canonical factorization of $C$ for a unique unitary $U \in \clb(H^2(\D^d))$, and let $S_i:= C M_{z_i} C$, $i=1, \ldots, d$. For each $\bk \in \Z^d_+$, we define
\[
\tilde{z}_{\bk} := C z^{\bk} = U z^{\bk}.
\]
Clearly, $\{\tilde z^{\bk}\}_{\bk \in \Z^d_+} \in B_{H^2(\D^d)}$. Also
\[
\begin{split}
S_{i} \tilde{z}^{\bk} & = CM_{z_{i}}C (C z^{\bk})
\\
& = C (z^{\bk + \epsilon_i})
\\
& = \tilde z^{\bk + \epsilon_i},
\end{split}
\]
where $\epsilon_i \in \Z_+^d$ is the multi-index with zero everywhere but $1$ in the $i$-th slot, and $i = 1, \ldots, d$. In other words, $(S_1, \ldots, S_d)$ forms a $d$-unilateral shift on $H^2(\D^d)$ corresponding to $\{\tilde z^{\bk}\}_{\bk \in \Z^d_+} \in B_{H^2(\D^d)}$. Next, assume that
\[
\vp = \sum_{\bk \in \Z^d} \vp_{\bk} z^{\bk} \in L^\infty(\T^d).
\]
As in the proof of \eqref{eqn: Toepl coeff}, it follows that
\[
\vp_{\bk - \bl} = \inp{T_{\vp} \z^{\bl}}{\z^{\bk}} \qquad (\bk, \bl \in \Z_+^d),
\]
and in view of $\{\tilde \z^{\bk}\}_{\bk \in \Z_+^d} \in B_{H^2(\D^d)}$, we define
\[
\tilde{\vp}_{\bk - \bl} := \langle T_\vp \tilde z^{\bl}, \tilde z^{\bk} \rangle  \qquad (\bk, \bl \in \Z_+^d).
\]
We are now ready for characterizations of symmetric Toeplitz operators on $H^2(\D^d)$. Along with the tools described above, the proof is similar to that of Corollary \ref{cor: phi and u}.

\begin{theorem}\label{thm: S toep n var}
Let $\vp = \sum_{\bk \in \Z_+^d} \vp_{\bk} z^{\bk} \in L^{\infty}(\T^d)$, and let $C$ be a conjugation with canonical factorization $C = UJ_{H^2(\D)}$. The following are equivalent:
\begin{enumerate}
\item $T_{\vp}$ is $C$-symmetric.

\item $\vp_{\bk - \bl} = \tilde \vp_{\bl - \bk}$ for all $\bk, \bl \in \Z_+^d$.

\item $\vp_{\bk - \bl} =  \inp{U^*T_{\vp}U z^{\bk}}{\z^{\bl}}$  for all $\bk, \bl \in \Z_+^d$.
\end{enumerate}
\end{theorem}
\begin{proof}
By definition, $T_{\vp}\in \clb(H^2(\D^d))$ is $C$-symmetric if and only if $CT_{\vp}C = T^*_{\vp}$. Therefore, $T_\vp$ is $C$-symmetric if and only if for all $\bk, \bl \in \Z_+^d$, we have
\[
\inp{CT_{\vp}C \z^{\bk} }{\z^{\bl}} = \inp{T^*_{\vp}\z^{\bk}}{\z^{\bl}},
\]
equivalently
\[
\inp{C z^{\bl}}{T_{\vp}C z^{\bk}} = \inp{z^{\bk}}{T_{\vp}z^{\bl}},
\]
as $C$ is a symmetry. This equality is further equivalent to the condition that
\[
\inp{T_{\vp}C z^{\bk}}{Cz^{\bl}} = \inp{T_{\vp}z^{\bl}}{z^{\bk}},
\]
that is
\[
\inp{T_{\vp} \tilde z^{\bk}}{ \tilde z^{\bl}} = \inp{T_{\vp} z^{\bl}}{z^{\bk}},
\]
and consequently, $T_\vp$ is $C$-symmetric if and only if $\vp_{\bk - \bl} = \tilde \vp_{\bl - \bk}$ for all $\bk, \bl \in \Z_+^d$. Finally, as in \eqref{eqn: Ustar T U}, $T_\vp$ is $C$-symmetric if and only if
\[
U^* T_\vp U = J_{H^2(\D^d)} T_\vp^* J_{H^2(\D^d)}.
\]
which, as in the proof of Corollary \ref{cor: phi and u}, is further equivalent to the condition that $\vp_{\bk - \bl} =  \inp{U^*T_{\vp}U z^{\bk}}{\z^{\bl}}$  for all $\bk, \bl \in \Z_+^d$. This completes the proof of the theorem.
\end{proof}

As an application of the above result, in Section \ref{sec: examples}, we will study a class of conjugations along the lines of Example \ref{illustration_example}. We refer the reader to \cite{Wang} for complex symmetric operators on the open unit ball in $\C^n$.

\section{Composition operators}\label{sec: composition}

In this example-based section, we apply our results to a nontrivial class of symmetric operators. Essentially, we connect complex symmetric Toeplitz operators and a special class of composition operators \cite{Shapiro book}.

Let $\vp = \sum_{n=-\infty}^{\infty} \vp_n z^n \in L^{\infty}(\T)$. Let $\theta$ be a holomorphic self-map of $\D$, and let $\psi \in H^{\infty}(\D)$. The \textit{weighted composition operator} $W_{\psi,\theta}$ with weight $\psi$ is defined by (cf. \cite{Jamison}) $W_{\psi,\theta}f = \psi \cdot (f \circ \theta)$, $f \in H^2(\D)$. That is
\[
W_{\psi,\theta}f(z)=\psi(z){f({\theta(z)})} \qquad (f \in H^2(\D), \; z\in \mathbb{D}).
\]
By \cite[Theorem 2.11]{Lim-Khoi}, $W_{\psi,\theta}$ is $J_{H^2(\D)}$-symmetric if and only if either of the following hold:
\begin{enumerate}
\item There exist $\alpha,\lambda \in \T$ such that $\psi(z) = \lambda$ and $\theta(z) = \alpha z$ for all $z\in \mathbb{D}$.

\item \label{conjugationType2} There exist $\alpha\in \mathbb{D}\setminus\{0\}$ and $\lambda \in \T$ such that $\psi(z)=\lambda \frac{\sqrt{1-|\alpha|^2}}{1-z\overline{\alpha}}$ and $\theta(z)=\frac{\bar{\alpha}}{\alpha}\frac{\alpha-z}{1-\bar{\alpha}z}$ for all $z\in \mathbb{D}$.
\end{enumerate}

We consider the nontrivial case, that is, we fix $\alpha\in \D \setminus \{0\}$ and define (by assuming that $\lambda = 1$)
\[
\psi(z) = \frac{\sqrt{1-|\alpha|^2}}{1-z\bar{\alpha}}, \text{ and }   \theta(z)=\frac{\bar{\alpha}}{\alpha}\frac{\alpha-z}{1-\bar{\alpha}z} \qquad (z \in \D).
\]
Define the (Szeg\"{o}) kernel function $k_\alpha \in H^2(\D)$ by
\[
k_{\alpha}(z)= \frac{1}{1-\bar{\alpha}z} \qquad (z \in \D).
\]
Since $\|k_{\alpha}\| = \frac{1}{\sqrt{1-|\alpha|^2}}$, we have
\[
\psi(z) = \frac{1}{\|k_{\alpha}\|} k_{\alpha}(z) \qquad (z \in \D).
\]
Since $W_{\psi,\theta}$ is $J_{H^2(\D)}$-symmetric, by Proposition \ref{prop: U is symmetric}, we conclude that
\begin{equation}\label{eqn: C psi theta}
C_{\psi,\theta} := W_{\psi,\theta} J_{H^2(\D)},
\end{equation}
defines a conjugation on $H^2(\D)$. Clearly
\[
C_{\psi,\theta}f(z)=\psi(z)\overline{f(\overline{\theta(z)})} \qquad (z \in \D, f \in H^2(\D)).
\]
Since $\theta(z)=\frac{\bar{\alpha}}{\alpha}\frac{\alpha-z}{1-\bar{\alpha}z}$, $z \in \D$, is a Blaschke factor, it follows that $M_\theta$ is an isometry and
\[
\ker M_\theta^* = \C k_\alpha.
\]
Throughout the sequel, $C_{\psi,\theta}$ will denote the conjugation on $H^2(\D)$ as defined in \eqref{eqn: C psi theta}. We need a lemma.

\begin{lemma}\label{Lemma WT Toeplitz}
If $T \in \clb(H^2(\D))$ is a Toeplitz operator, then $W^*_{\psi,\theta}TW_{\psi,\theta}$ is also a Toeplitz operator.
\end{lemma}
\begin{proof}
Let us first verify that
\[
W_{\psi,\theta}M_{z} = M_{\theta} W_{\psi,\theta}.
\]
For each $f \in H^2(\D)$, we have
\[
\begin{split}
W_{\psi,\theta} M_{z}f &= W_{\psi,\theta} (zf)
\\
& = \psi(z) \theta(z) f({\theta(z)})
\\
& = \theta(z) \psi(z) {f({\theta(z)})}
\\
& = M_{\theta}W_{\psi,\theta} f,
\end{split}
\]
which completes the proof of the claim. Now let $\vp \in L^\infty(\T)$. Then
\[
M^*_{z} (W^*_{\psi,\theta}T_{\vp}W_{\psi,\theta}) M_z = W^*_{\psi,\theta} M^*_{\theta}T_{\vp} M_{\theta} W_{\psi,\theta}.
\]
Since $\theta$ is an inner function, by the Brown-Halmos criterion \cite{BH}, $T_{\vp}$ is $M_{\theta}$-Toeplitz, that is, $M^*_{\theta} T_{\vp} M_{\theta} = T_{\vp}$. The above equality then implies
\[
M^*_{z} (W^*_{\psi,\theta}T_{\vp}W_{\psi,\theta}) M_z = W^*_{\psi,\theta}T_{\vp}W_{\psi,\theta},
\]
and completes the proof of the lemma.
\end{proof}

We apply this to symmetric Toeplitz operators:

\begin{proposition}\label{Top-W-Symmetry}
Let $\varphi = \sum_{n=-\infty}^{\infty} \vp_n z^n \in L^\infty(\T)$. Then $T_\varphi$ is $C_{\psi,\theta}$-symmetric if and only if
\[
{\varphi}_{n}=\inp{ T_\varphi W_{\psi,\theta}z^n}{W_{\psi,\theta}1} \mbox{ and } {\varphi}_{-n}=\inp{ T_\varphi W_{\psi,\theta}1}{W_{\psi,\theta}z^{n}},
\]
for all $n \geq 1$.
\end{proposition}
\begin{proof}
Note that, by definition, $C_{\psi,\theta} = W_{\psi,\theta} J_{H^2(\D)}$ the canonical factorization of $C_{\psi,\theta}$. By Corollary \ref{cor: phi and u}, $T_\varphi$ is $C_{\psi,\theta}$-symmetric if and only if
\begin{equation}\label{eqn: phi theta psi}
{\varphi}_{n-m}=\inp{ T_\varphi W_{\psi,\theta} z^n}{W_{\psi,\theta} z^m} \qquad (m, n \in \Z_+).
\end{equation}
In particular, if $T_\varphi$ is $C_{\psi,\theta}$-symmetric, then the conditions hold. For the converse, we first note, in view of Lemma \ref{Lemma WT Toeplitz}, that $W^*_{\psi,\theta}TW_{\psi,\theta}$ is a Toeplitz operator. Suppose $n \geq 1$. For all $i, j \in \Z_+$ such that $n = i-j$, we have
\[
\inp{ T_\varphi W_{\psi,\theta}z^i}{W_{\psi,\theta}z^j} = \inp{ T_\varphi W_{\psi,\theta}z^n}{W_{\psi,\theta}1},
\]
and hence, by assumption
\[
{\vp}_{n} = \inp{ T_\varphi W_{\psi,\theta}z^n}{W_{\psi,\theta}z^0}.
\]
Therefore
\[
{\vp}_{i-j} = \inp{T_\varphi W_{\psi,\theta}z^i}{W_{\psi,\theta}z^j} \qquad  (i\geq j \geq 1).
\]
Similarly, by using the second condition ${\varphi}_{-n}=\inp{ T_\varphi W_{\psi,\theta}1}{W_{\psi,\theta}z^{n}}$, we find
\[
{\vp}_{i-j} =   \inp{ T_\varphi W_{\psi,\theta}z^i}{W_{\psi,\theta}z^j} \qquad  (j\geq i \geq 1).
\]
The above two equalities yield
\[
{\vp}_{i-j} =   \inp{ T_\varphi W_{\psi,\theta}z^i}{W_{\psi,\theta}z^j} \qquad  (i, j \geq 1),
\]
which implies $T_{\vp}$ is $C_{\psi,\theta}$-symmetric and completes the proof.
\end{proof}

Recall the third part of Corollary \ref{cor: phi and u}: For $\vp = \sum_{n=-\infty}^{\infty} \vp_n z^n \in L^\infty(\T)$ and a conjugation $C = U J_{H^2(\D)}$ on $H^2(\D)$, that $T_\vp$ is $C$-symmetric if and only if
\[
\vp_{m-n} = \sum_{i,j=0}^{\infty} u_{m,j} \vp_{i-j} \overline{u_{i,n}} \qquad (m,n \in \Z_+),
\]
where
\[
u_{i,j} := \langle U z^i, z^j \rangle = \langle C z^i, z^j \rangle \qquad (i,j \in \Z_+).
\]
In the following, we compare this with $C_{\psi,\theta}$-symmetric Toeplitz operators. Clearly, we need to compute $u_{i,j}$ for all $i,j \in \Z_+$. First, observe that
\[
W_{\psi, \theta} z^n = \theta(z)^{n} \frac{k_{\alpha}}{\|k_{\alpha}\|} \qquad (n \in \Z_+).
\]
By recalling that $\theta(z)=\frac{\bar{\alpha}}{\alpha}\frac{\alpha-z}{1-\bar{\alpha}z}$, $z \in \D$, we compute $\theta(z)^n k_{\alpha}$ as
\[
\begin{split}
\theta(z)^n k_{\alpha} & =\left(\frac{\bar{\alpha}}{\alpha}\right)^n\frac{(\alpha-z)^n}{(1-\bar{\alpha}z)^{n+1}}
\\
& = \bar{\alpha}^n \left(1-\frac{z}{\alpha}\right)^n (1-\bar{\alpha}z)^{-(n+1)}
\\
& = \bar{\alpha}^n \displaystyle\sum_{p=0}^{n}\frac{(-1)^p}{\alpha^p}\left(\begin{array}{l}n \\ p\end{array}\right)z^p \sum_{q=0}^{\infty}\bar{\alpha}^q\left(\begin{array}{c}n+q \\ q\end{array}\right)z^q.
\end{split}
\]
For each $p, q \geq 0$, set
\[
\beta_p=\frac{(-1)^p}{\alpha^p}\left(\begin{array}{l}n \\ p\end{array}\right), \text{ and } \gamma_q=\bar{\alpha}^q\left(\begin{array}{c}n+q \\ q\end{array}\right).
\]
It follows that
\[
\begin{split}
\theta(z)^n k_{\alpha} & = \bar{\alpha}^n \displaystyle\sum_{p=0}^{n}\beta_pz^p \sum_{q=0}^{\infty}\gamma_qz^q
\\
& = \bar{\alpha}^n \displaystyle\sum_{r=0}^{\infty}c_rz^r,
\end{split}
\]
where
\[
c_r=\displaystyle\sum_{m=0}^{\min\{n,r\}}\beta_m\gamma_{r-m} \qquad (r \in \Z_+).
\]
Now we are ready to compute $u_{i,j}$. Let $i, j \in \Z_+$. Then
\[
\begin{split}
u_{i,j} & = \langle W_{\psi,\theta}z^i,z^j\rangle
\\
& = \sqrt{1-|\alpha|^2} \bar{\alpha}^i \displaystyle\sum_{m=0}^{\min\{i,j\}}\beta_m\gamma_{j-m}
\\
& = \bar{\alpha}^i  \sqrt{1-|\alpha|^2} \displaystyle\sum_{m=0}^{\min\{i,j\}}\frac{(-1)^m}{\alpha^m}\left(\begin{array}{l}i \\ m\end{array}\right)\bar{\alpha}^{(j-m)}\left(\begin{array}{c}i+j-m \\ j-m\end{array}\right)
\\
& = \sqrt{1-|\alpha|^2}\displaystyle\sum_{m=0}^{\min\{i,j\}}\frac{(-1)^m}{\alpha^m}\left(\begin{array}{l}i \\ m\end{array}\right)\bar{\alpha}^{(i+j-m)}\left(\begin{array}{c}i+j-m \\ k-m\end{array}\right)
\\
& = \sqrt{1-|\alpha|^2}\displaystyle\sum_{m=0}^{\min\{i,j\}}\frac{(-1)^m}{\alpha^m}\left(\begin{array}{l}i \\ m\end{array}\right)\bar{\alpha}^{(i+j-m)}\left(\begin{array}{c}i+j-m \\ i\end{array}\right).
\end{split}
\]
Therefore
\[
u_{i,j} = \sqrt{1-|\alpha|^2}\displaystyle\sum_{m=0}^{\min\{i,j\}}\frac{(-1)^m}{\alpha^m}\left(\begin{array}{l}i \\ m\end{array}\right)\bar{\alpha}^{(i+j-m)}\left(\begin{array}{c}i+j-m \\ i\end{array}\right),
\]
for all $i, j \in \Z_+$. In particular, we have
\begin{equation}\label{special terms}
\inp{W_{\psi,\theta} 1}{z^j} = \sqrt{1-|\alpha|^2}\bar{\alpha}^j.
\end{equation}
This also follows directly from the fact that $\inp{W_{\psi,\theta} 1}{z^j} = \sqrt{1-|\alpha|^2} \langle k_\alpha,z^j\rangle$. Recall also from \eqref{eqn: u twist} that $u_{i,j} = u_{j,i}$. Therefore, we have proved:

\begin{proposition}\label{Top-W-Symmetry2}
Let $\vp = \sum_{n=-\infty}^{\infty} \vp_n z^n \in L^\infty(\T)$. Then $T_\vp$ is $C_{\psi,\theta}$-symmetric if and only if
\[
\vp_{m-n} = \sum_{i,j=0}^{\infty} u_{m,j} \vp_{i-j} \overline{u_{i,n}} \qquad (m,n \in \Z_+),
\]
where
\[
u_{i,j} = \sqrt{1-|\alpha|^2}\displaystyle\sum_{m=0}^{\min\{i,j\}}\frac{(-1)^m}{\alpha^m}\left(\begin{array}{l}i \\ m\end{array}\right)\bar{\alpha}^{(i+j-m)}\left(\begin{array}{c}i+j-m \\ i\end{array}\right),
\]
for all $i, j \in \Z_+$.
\end{proposition}

Now we consider a special case: Let
\[
\varphi(z) = {\varphi}_{-1} \bar{z} + {\varphi}_0 + {\varphi}_1z \in L^\infty(\T).
\]
By Proposition \ref{Top-W-Symmetry2}, $T_\varphi$ is $C_{\psi,\theta}$-symmetric if and only if
\begin{equation}\label{eq: sec 8}
{\vp}_{1} = \sum_{i,j} u_{1,j} {\vp}_{i-j}\overline{u_{i,0}}, \text{ and }
{\vp}_{-1} = \sum_{i,j} u_{0,j} {\vp}_{i-j}\overline{u_{i,1}},
\end{equation} 	
and
\[
{\vp}_{0} = \sum_{ij} u_{0,j} {\vp}_{i-j}\overline{u_{i,0}}.
\]
In view of \eqref{special terms}, the latter condition can be further simplified. Indeed, we have
\[
{\varphi}_0 = {\varphi}_{-1}(\sum_{k=0}^{\infty} (1-|\alpha|^2)\bar{\alpha}^{k+1}\alpha^k) + {\varphi}_0 + {\varphi}_1 (\sum_{k=0}^{\infty}| (1-|\alpha|^2)\bar{\alpha}^k\alpha^{k+1}),
\]
and hence
\[
\begin{split}
0 & = {\varphi}_{-1}(\sum_{k=0}^{\infty}| (1-|\alpha|^2)\bar{\alpha}^{k+1}\alpha^k) + {\varphi}_1 (\sum_{k=0}^{\infty} (1-|\alpha|^2)\bar{\alpha}^k\alpha^{k+1})
\\
& = (\sum_{k=0}^{\infty} (1-|\alpha|^2) |\alpha|^{2k}) (\bar{\alpha} \vp_{-1} + \alpha \vp_1)
\\
& = \bar{\alpha} \vp_{-1} + \alpha \vp_1.
\end{split}
\]
Therefore, ${\vp}_{0} = \sum_{ij} u_{0,j} {\vp}_{i-j}\overline{u_{i,0}}$ if and only if
\[
{\varphi}_{1} = - \frac{\bar{\alpha}}{\alpha} {\varphi}_{-1}.
\]
Moreover, the first equality in \eqref{eq: sec 8} implies
\[
\varphi_{1}=\varphi_{-1}\sum_{k=0}^{\infty}u_{1,k+1}\bar{u}_{0,k}+\varphi_0\delta_{1,0}+\varphi_1\sum_{k=0}^{\infty}u_{1,k}\bar{u}_{0,k+1},
\]
that is
\begin{equation}\label{eqn: comp exam 1}
\varphi_{1} = \sqrt{1-|\alpha|^2}\left[ \varphi_{-1}\sum_{k=0}^{\infty}\alpha^k u_{1,k+1}+\varphi_1\sum_{k=0}^{\infty}\alpha^{k+1}u_{1,k}\right],
\end{equation}
whereas the second equality in \eqref{eq: sec 8} yields
\[
\varphi_{-1} =\varphi_{-1}\sum_{k=0}^{\infty}u_{0,k+1}\bar{u}_{1,k}+\varphi_0\delta_{0,1}+\varphi_1\sum_{k=0}^{\infty}u_{0,k}\bar{u}_{1,k+1},
\]
that is
\begin{equation}\label{eqn: comp exam 2}
\varphi_{-1} = \sqrt{1-|\alpha|^2} \left[\varphi_{-1} \sum_{k=0}^{\infty}\bar{\alpha}^{k+1}\bar{u}_{1,k}+\varphi_1\sum_{k=0}^{\infty}\bar{\alpha}^k\bar{u}_{1,k+1} \right].
\end{equation}
This implies that $T_\varphi$ is $C_{\psi,\theta}$-symmetric if and only if ${\varphi}_{1} = - \frac{\bar{\alpha}}{\alpha} {\varphi}_{-1}$ and both \eqref{eqn: comp exam 1} and \eqref{eqn: comp exam 2} hold.

\noindent Now, if possible, suppose that $T_\varphi$ is $C_{\psi,\theta}$-symmetric. Plugging ${\varphi}_{1} = - \frac{\bar{\alpha}}{\alpha} {\varphi}_{-1}$ into \eqref{eqn: comp exam 1} and then cancelling the factor $\vp_1$ out from each side, we get
\[
\begin{split}
1 & = \sqrt{1-|\alpha|^2}\left[ -\frac{\alpha}{\bar{\alpha}} \sum_{k=0}^{\infty}\alpha^k u_{1,k+1} + \sum_{k=0}^{\infty}\alpha^{k+1}u_{1,k}\right]
\\
& = \sqrt{1-|\alpha|^2}\left[ -\frac{1}{\bar{\alpha}}\sum_{k=0}^{\infty}\alpha^{k+1} u_{1,k+1}+\sum_{k=0}^{\infty}\alpha^{k+1}u_{1,k}\right]
\\
&=\sqrt{1-|\alpha|^2}\sum_{k=0}^{\infty}\alpha^{k+1}\left[u_{1,k} -\frac{1}{\bar{\alpha}}u_{1,k+1} \right],
\end{split}
\]
and hence
\begin{equation}\label{eqn: proof example 2}
1 = \sqrt{1-|\alpha|^2}\sum_{k=1}^{\infty}\alpha^{k+1}\left[u_{1,k} -\frac{1}{\bar{\alpha}}u_{1,k+1} \right] +\sqrt{1-|\alpha|^2}\alpha\left[ u_{1,0} -\frac{1}{\bar{\alpha}}u_{1,1}\right].
\end{equation}
Now we compute $u_{1,k} -\frac{1}{\bar{\alpha}}u_{1,k+1}$, $k \geq 0$. Set, for the shake of simplicity, $\beta = \sqrt{1-|\alpha|^2}$. Recall 	
\[
u_{i,j}= \beta \displaystyle\sum_{m=0}^{\min\{i,j\}}\frac{(-1)^m}{\alpha^m}\left(\begin{array}{l}i \\ m\end{array}\right)\bar{\alpha}^{(i+j-m)}\left(\begin{array}{c}i+j-m \\ i\end{array}\right).
\]
Therefore, for each $k \geq 1$, we have
\[
u_{k,1}= \beta \displaystyle\sum_{m=0}^{\min\{k,1\}}\frac{(-1)^m}{\alpha^m}\left(\begin{array}{l}k \\ m\end{array}\right)\bar{\alpha}^{(k+1-m)}\left(\begin{array}{c}k+1-m \\ k\end{array}\right),
\]
which implies
\[
\begin{split}
u_{k,1}-\frac{1}{\bar{\alpha}}u_{1,k+1} & = \beta \left[\displaystyle\sum_{m=0}^{1}\frac{(-1)^m}{\alpha^m}\left(\begin{array}{l}k \\ m\end{array}\right)\bar{\alpha}^{(k+1-m)}\left(\begin{array}{c}k+1-m \\ k\end{array}\right)\right.\\ &\left. \hspace{3cm}-\displaystyle\sum_{m=0}^{1}\frac{(-1)^m}{\alpha^m}\left(\begin{array}{c}k+1 \\ m\end{array}\right)\bar{\alpha}^{(k+1-m)}\left(\begin{array}{c}k+2-m \\ k+1\end{array}\right)\right]
\\
& = \beta \left[\bar{\alpha}^{(k+1)} \left(\begin{array}{c}k+1 \\ k\end{array}\right)-\frac{1}{\alpha}\left(\begin{array}{c}k \\ 1\end{array}\right)\bar{\alpha}^{k}-\bar{\alpha}^{(k+1)}\left(\begin{array}{c}k+2 \\ k+1\end{array}\right)+\frac{1}{\alpha}\left(\begin{array}{c}k+1 \\ 1\end{array}\right)\bar{\alpha}^{k}  \right]
\\
&= \beta \left[\frac{\bar{\alpha}^k}{\alpha}-\bar{\alpha}^{(k+1)}\right]\\
& = (1-|\alpha|^2)^{3/2}\frac{\bar{\alpha}^k}{\alpha}.
\end{split}
\]
On the other hand
\[
 u_{1,0} -\frac{1}{\bar{\alpha}}u_{1,1}=\frac{(1-|\alpha|^2)^{3/2}}{\alpha}.
\]	
Consequently, \eqref{eqn: proof example 2} yields
\[
\begin{split}
1&=\sqrt{1-|\alpha|^2}\sum_{k=1}^{\infty}\alpha^{k+1}\left[(1-|\alpha|^2)^{3/2}\frac{\bar{\alpha}^k}{\alpha} \right] +\sqrt{1-|\alpha|^2}\alpha\left[ \frac{(1-|\alpha|^2)^{3/2}}{\alpha}\right]\\
&=(1-|\alpha|^2)^2\sum_{k=1}^{\infty}|\alpha |^{2k} +\sqrt{1-|\alpha|^2}\alpha\left[ \frac{(1-|\alpha|^2)^{3/2}}{\alpha}\right]\\
&=(1-|\alpha|^2)^2\left(\sum_{k=1}^{\infty}|\alpha |^{2k}+1 \right)\\
&=1-|\alpha|^2,
\end{split}
\]
which contradicts the fact that $\alpha \neq 0$. Thus we have proved:

\begin{proposition}
Let
\[
\varphi(z) = {\varphi}_{-1} \bar{z} + {\varphi}_0 + {\varphi}_1z \in L^\infty(\T).
\]
Then $T_\vp$ is $C_{\psi,\theta}$-symmetric if and only if $\vp$ is a constant function.
\end{proposition}

\section{Examples}\label{sec: examples}

In this section, we present more examples of symmetric Toeplitz operators and comment on some of our results and methodology. We begin with the easy case: finite Toeplitz matrices. Recall that a Toeplitz matrix of order $N+1$, $N \geq 1$, admits the following representation
\begin{equation}\label{finite_Toeplitz_matrix}
T=
\begin{bmatrix}
a_0 & a_{-1}  & a_{-2}  & \cdots & a_{-N}
\\
a_1    & a_0     & a_{-1}  & \ddots & a_{-N+1}
\\
a_2 & a_1 & a_0  & \ddots & a_{-N+2}
\\
\vdots & \ddots & \ddots & \ddots & \vdots
\\
a_{N}& a_{N-1}  & a_{N-2}  & \cdots   & a_0
\end{bmatrix}.
\end{equation}
Consider the standard orthonormal basis of $\C^{N+1}$ as
\[
\zeta = \{e_n\}_{n \in \Lambda} \in B_{\C^{N+1}},
\]
where
\[
\Lambda = \{0,1,\ldots, N\}.
\]
Suppose $C$ is a conjugation on $\C^{N+1}$ corresponding to $\{f_n\}_{n \in \Lambda} \in B_{\C^{N+1}}$, that is, $Cf_n = f_n$ for all $n \in \Lambda$ (see Definition \ref{def: conj corres}). By Proposition \ref{conju_Charac_theorem}, we have
\begin{equation}\label{finite_dim_conju}
C(\sum_{n=0}^N a_n e_n) = \sum_{n=0}^N \sum_{m=0}^N \bar{a}_n c_{n,m}^{(\zeta)} e_m,
\end{equation}
for all $(a_0, a_1, \ldots, a_N) \in \C^{N+1}$, where
\[
c_{n,m}^{(\zeta)}=\displaystyle\sum_{k=0}^N \langle f_k, e_n\rangle \langle f_k, e_m\rangle,
\]
for all $m, n \in \Lambda$. We are now ready for characterizations of $C$-symmetric Toeplitz matrices.

\begin{theorem}\label{finiteToeplitzCS}
Let $T$ be a Toeplitz matrix, and let $C$ be a conjugation on $\C^{N+1}$ as in \eqref{finite_Toeplitz_matrix} and \eqref{finite_dim_conju}, respectively. Then $T$ is $C$-symmetric if and only if
\[
\sum_{m=0}^{N}c^{(\zeta)}_{m,p} \bar{a}_{m-k} = \sum_{m=0}^{N}c^{(\zeta)}_{m,k}\bar{a}_{m-p},
\]
for all $k, p \in \Lambda$.
\end{theorem}
\begin{proof}
Note that $T$ is $C$-symmetric with $C$ as in \eqref{finite_dim_conju} if and only if
\[
CTe_k=T^*Ce_k,
\]
for all $k \in \Lambda$. The proof now follows in a manner similar to the proof of Theorem \ref{thm_main}.
\end{proof}

We illustrate this with two simple examples of conjugations. Let $f_n = e_n$ for all $n \in \Lambda$. Then $C$ becomes the standard canonical conjugation $J_{\C^{N+1}}$, where
\[
J_{\mathbb{C}^{N+1}}(z_0,z_1,z_2,\dots,z_N) = (\bar{z}_0,\bar{z}_1,\bar{z}_2,\dots,\bar{z}_N),
\]
for all $(z_0,z_1,z_2,\dots,z_N) \in \C^{N+1}$. In this case, we have
\[
c_{m,n}^{(\zeta)} = \begin{cases}
1 & \mbox{if } m=n
\\
0 & \mbox{otherwise},
\end{cases}
\]
which, along with Theorem \ref{finiteToeplitzCS}, implies: the Toeplitz matrix $T$ in \eqref{finite_Toeplitz_matrix} is $J_{\mathbb{C}^{N+1}}$-symmetric if and only if
\[
a_{p-k}=a_{k-p} \qquad (k,p \in \Lambda).
\]
Of course, this follows straight from the definition of symmetric operators. Next, we verify the above corresponding to the Toeplitz conjugation $C_{\text Toep}$ on $\C^{N+1}$. Recall from \eqref{eqn: can con FD} that
\[
C_{\text Toep} (z_0, z_1, \ldots, z_N) = (\overline{z}_{N}, \overline{z}_{N-1}, \ldots, \overline{z}_0),
\]
for all $(z_0,z_1, \ldots, z_N) \in \C^{N+1}$. In this case, we have
\[
\begin{split}
c^{(\zeta)}_{m,n} & =\langle C_{\text Toep} e_m,e_n\rangle
\\
& =\langle e_{N-m},e_n\rangle
\\
& = \begin{cases}
1 & \mbox{if } m+n=N\\
0 & \text{otherwise}.
\end{cases}
\end{split}
\]
This implies
\[
\sum_{m=0}^{N}c^{(\zeta)}_{m,p}\bar{a}_{m-k}=\bar{a}_{N-(p+k)},
\]
as well as
\[
\sum_{m=0}^{N}c^{(\zeta)}_{m,k}\bar{a}_{m-p}=\bar{a}_{N-(p+k)},
\]
for all $k,p \in \Lambda$. Theorem \ref{finiteToeplitzCS} then recovers the well-known fact that:

\begin{corollary}\label{cor: matrix symm}
Toeplitz matrices are $C_{\text Toep}$-symmetric.
\end{corollary}

The following example illustrates Proposition \ref{thm: symm S Toep} in view of Example \ref{illustration_example}:

\begin{example}
Recall from Example \ref{illustration_example}, for each $\theta, \xi \in \R$, that $C_{\theta, \xi}$ is a conjugation on $H^2(\D)$ with respect to the basis $\{e^{\frac{i\xi}{2}} e^{\frac{-i n \theta}{2}}z^n\}_{n \in \Z_+} \in B_{H^2(\D)}$, where
\[
(C_{\theta, \xi} f)(z)=e^{i\xi}\overline{f(e^{i\theta}\bar{z})} \qquad (f \in H^2(\D)).
\]
If we set
\[
S_{\theta, \xi} = C_{\theta, \xi} M_z C_{\theta, \xi},
\]
then a simple calculation reveals that
\[
S_{\theta, \xi} = e^{-i\theta}M_z.
\]
Suppose $T_{\vp}$, $\vp \in L^\infty(\T)$, is a $C_{\xi,\theta}$-symmetric Toeplitz operator. Then Proposition \ref{thm: symm S Toep} implies that $T_{\vp}$ is a $S_{\theta, \xi}$-Toeplitz operator.
\end{example}

In the context of the counterexample following Proposition \ref{thm: symm S Toep}, we now exhibit an example of conjugation $C \notin \{M_z\}'$ and a symbol $\vp \in L^\infty(\T)$ such that $T_\vp$ is not $C$-symmetric.

\begin{example}
Consider the conjugation $C$ on $H^2(\D)$ defined by
\[
C(\sum_{n=0}^{\infty} a_n z^n) = \bar{a}_0 + \bar{a}_2 z + \bar{a}_1 z^2 + \sum_{n=3}^{\infty} \bar{a}_n z^n,
\]
for all $\sum_{n=0}^{\infty}a_n z^n \in H^2(\D)$. Note that $C(1) = 1$, $Cz^2=z$, and $Cz^n = z^{n+1}$ for $n=1$ and all $n \geq 3$. Moreover, $\{f_n\}_{n \in \Z_+} \in B_{H^2(\D)}$ and $Cf_n = f_n$ for all $n \in \Z_+$, where
\[
f_n = \begin{cases}
1 & \mbox{if } n=0 \\
\frac{1}{\sqrt{2}} (z^{2n-1} + z^{2n}) & \mbox{otherwise}.
\end{cases}
\]
Clearly
\[
M_zC \neq C M_z.
\]
In this case, for the Toeplitz operator $T_\vp$, with
\[
\vp(z)=i\bar{z}^2+\bar{z}-z-iz^2 \qquad (z \in \T),
\]
it follows that $C T_\vp \neq T_\vp^*C$, that is, $T_\vp$ is not $C$-symmetric.
\end{example}

We conclude this section with Toeplitz operators on $H^2(\D^d)$. We follow the one variable construction described in Example \ref{illustration_example}. Recall that (in view of the identification of $H^2(\mathbb{D}^d)$ as a closed subspace via radial limits)
\[
\{z^{\bk}\}_{\bk \in \Z_+^d} \in B_{H^2(\D^d)} \text{ and } \{z^{\bk}\}_{\bk \in \Z^d} \in B_{L^2(\T^d)}.
\]
Fix $\bm{\theta} = (\theta_1, \ldots, \theta_d)$ and $\bm{\xi} = (\xi_1, \ldots, \xi_d)$ in $\R^d$, and set
\[
f_{\bk} = \exp\Big(\frac{i}{2} \sum_{j=1}^d \xi_j\Big) \exp\Big(\frac{-i}{2} \sum_{j=1}^d k_j \theta_j\Big) z^{\bk} \qquad (\bk \in \Z_+^d).
\]
Since
\[
\langle f_{\bk}, f_{\bl} \rangle = \exp\Big(\frac{i}{2} \sum_{j=1}^d (l_j - k_j) \theta_j \Big) \langle z^{\bk}, z^{\bl} \rangle,
\]
for all $\bk, \bl \in \Z_+^d$ and $\{e_{\bk}\}_{\bk \in \Z_+^d} \in B_{H^2(\D^d)}$, it follows that
\[
\{f_{\bk}\}_{\bk \in \Z_+^d} \in B_{H^2(\D^d)}.
\]
Denote by $C_{\bm{\theta}, \bm{\xi}}$ the conjugation corresponding to $\{f_{\bk}\}_{\bk \in \Z_+^d} \in B_{H^2(\D^d)}$, that is
\[
C_{\bm{\theta}, \bm{\xi}} f_{\bk} = f_{\bk} \qquad (\bk \in \Z_+^d).
\]
We now proceed to compute the representation of $C_{\bm{\theta}, \bm{\xi}}$ with respect to the canonical basis $\zeta = \{z^{\bk}\}_{\bk \in \Z_+^d} \in B_{H^2(\D^d)}$. In view of Corollary \ref{cor: rep of C}, we have
\[
c_{\bl, \bm{m}}^{(\zeta)} = \sum_{\bk \in \Z_+^d}\langle f_{\bk}, z^{\bl} \rangle \langle f_{\bk}, z^{\bm{m}}\rangle
\]
and hence
\[
c_{\bk, \bm{m}}^{(\zeta)}  =
\begin{cases}
0 & \text{if } \bm{m} \neq \bk \\
\exp\Big(i \sum_{j=1}^d \xi_j\Big) \exp\Big(-i \sum_{j=1}^d k_j \theta_j\Big) & \text{if } \bm{m}= \bk.
\end{cases}
\]
Let $f = \sum_{\bk \in \Z_+^d} a_{\bk} z^{\bk} \in H^2(\D^d)$. As in Example \ref{illustration_example}, we compute
\[
\begin{split}
C_{\bm{\theta}, \bm{\xi}}f & = \sum_{\bm{k} \in \Z_+^d} \sum_{\bm{m} \in \Z_+^d} \overline{a_{\bk}} c_{\bk, \bm{m}}^{(\zeta)} z^{\bm{m}}
\\
& = \sum_{\bm{k} \in \Z_+^d} \overline{a_{\bk}} c_{\bk, \bk}^{(\zeta)} z^{\bk}
\\
& = \exp\Big(i \sum_{j=1}^d \xi_j\Big) \sum_{\bm{k} \in \Z_+^d} \overline{a_{\bk}} \exp\Big(-i \sum_{j=1}^d k_j \theta_j\Big) z^{\bk}.
\end{split}
\]
Therefore, we have the following:

\begin{proposition}\label{prop: C on Dd}
For each $\bm{\theta} = (\theta_1, \ldots, \theta_d)$ and $\bm{\xi} = (\xi_1, \ldots, \xi_d)$ in $\R^d$, the map
\[
C_{\bm{\theta}, \bm{\xi}} (\sum_{\bk \in \Z_+^d} a_{\bk} e_{\bk}) = \exp\Big(i \sum_{j=1}^d \xi_j\Big) \sum_{\bm{k} \in \Z_+^d} \overline{a_{\bk}} \exp\Big(-i \sum_{j=1}^d k_j \theta_j\Big) z^{\bk},
\]
defines a conjugation on $H^2(\D^d)$.
\end{proposition}

Next, we characterize $C_{\bm{\theta}, \bm{\xi}}$-symmetric Toeplitz operators on $H^2(\D^d)$.

\begin{theorem}\label{example d var}
Let $\bm{\theta}, \bm{\xi} \in \R^d$, and let $\varphi = \sum_{\bk \in \Z^d} {\varphi}(\bk) z^{\bk} \in L^\infty(\mathbb{T}^d)$. Then $T_\varphi$ is $C_{\bm{\theta}, \bm{\xi}}$-symmetric if and only if
\[
\exp\Big(i \sum_{j=1}^d k_j \theta_j\Big) {\varphi}(\bk) = {\varphi}(-\bk) \qquad (\bk \in \Z^d).
\]
\end{theorem}
\begin{proof}
We proceed as follows: First we consider the canonical factorization of $C_{\bm{\theta}, \bm{\xi}} $ as  $C_{\bm{\theta}, \bm{\xi}}= U J_{H^2(\D^d)}$. Since $U( \z^{\bk})  =  C_{\bm{\theta}, \bm{\xi}} ( \z^{\bk})$, by Proposition \ref{prop: C on Dd}, it follows that
\[
U( \z^{\bk}) = \exp\Big(i \sum_{j=1}^d \xi_j\Big) \exp\Big(-i \sum_{j=1}^d k_j \theta_j\Big) \z^{\bk},
\]
that is
\[
U( \z^{\bk}) = \lambda \exp\Big(-i \sum_{j=1}^d k_j \theta_j\Big) \z^{\bk} \qquad (\bk \in \Z_+^d),
\]
where
\[
\lambda:= \exp\Big(i \sum_{j=1}^d \xi_j\Big).
\]
For $\bk, \bl \in \Z_+^d$, we compute
\[
\begin{split}
\inp{T_{\vp}U \z^{\bk}}{U\z^{\bl}} & = \inp{T_{\vp} \Big(\lambda \exp\Big(-i \sum_{j=1}^d k_j \theta_j\Big) \z^{\bk}\Big)}{\lambda  \exp\Big(-i \sum_{j=1}^d l_j \theta_j\Big) \z^{\bl}}
\\
& = |\lambda|^2  \exp\Big(-i \sum_{j=1}^d k_j \theta_j\Big) \times \overline{\exp\Big(-i \sum_{j=1}^d l_j \theta_j\Big)} ~~\inp{T_{\vp} \z^{\bk}}{  \z^{\bl}}
\\
& = \exp\Big(i (\sum_{j=1}^d (l_j-k_j) \theta_j\Big)\inp{T_{\vp} \z^{\bk}}{  \z^{\bl}}
\\
& = \exp\Big(i (\sum_{j=1}^d (l_j-k_j) \theta_j\Big) {\vp}_{\bl - \bk}.
\end{split}
\]
Then, by Theorem \ref{thm: S toep n var}, $T_\vp$ is $C_{\bm{\theta}, \bm{\xi}}$-symmetric if and only if
\[
\vp_{\bk - \bl} = \exp\Big(i (\sum_{j=1}^d (l_j-k_j) \theta_j\Big) {\vp}_{\bl - \bk},
\]
for all $\bk, \bl \in \Z_+^d$, or equivalently
\[
{\vp}_{\bm{n}} = \exp\Big(i (\sum_{j=1}^d (n_j) \theta_j\Big) {\vp}_{-\bm{n}} \qquad (\bm{n} \in \Z_+^d),
\]
which completes the proof of the theorem.
\end{proof}

\section{Appendix}\label{sec: appendix}

In this section, we prove some results on intertwiners that are not directly related to Toeplitz operators but fit well in the context of symmetric operators. Some of the results may be of independent interest. We begin with some elementary observations.

Recall that (see Section \ref{sec: S Toeplitz}) for $X \in \clb(H^2(\D))$, that $X = M_\theta$ for some $\theta \in H^\infty(\D)$ if and only if
\[
X M_z = M_z X.
\]
The anti-linear counter part of the above states:

\begin{proposition}
Let $X \in \cll_a(H^2(\mathbb{D}))$. Then $XM_z=M_zX$ if and only if $X= J_{H^2(\D)} M_\theta$ for some $\theta\in H^\infty(\D)$.
\end{proposition}
\begin{proof}
If $X= J_{H^2(\D)} M_\theta$, then $XM_z=M_zX$ follows from the fact that $M_z J_{H^2(\D)} = J_{H^2(\D)} M_z$. Now suppose that $XM_z=M_zX$. Then $J_{H^2(\D)} X M_z = J_{H^2(\D)} M_z X$, and hence, by $M_z J_{H^2(\D)} = J_{H^2(\D)} M_z$, it follows that
\[
(J_{H^2(\D)} X)M_z = M_z (J_{H^2(\D)}X).
\]
This implies that $J_{H^2(\D)} X = M_\theta$ for some $\theta \in H^\infty(\D)$, and hence $X = J_{H^2(\D)} M_\theta$.
\end{proof}

Similarly, one can prove: If $X \in \clc(H^2(\mathbb{D}))$, then $M^*_z X M_z = X$ if and only if $X= J_{H^2(\D)} T_\vp$ for some $\vp \in L^\infty(\T)$. In particular, if $C \in \clc(H^2(\D))$ and $CM_z = M_z C$, then $C = J_{H^2(\D)} M_\theta$, which implies
\[
M_\theta = C J_{H^2(\D)},
\]
and hence $M_\theta$ is a unitary. Therefore, $\theta \equiv \lambda$ for some $\lambda \in \mathbb{T}$, which implies that
\[
C = \lambda J_{H^2(\D)}.
\]
In Theorem \ref{Main:theorem}, we will generalize this observation in the setting of shifts of multiplicity one.

The following is a simple (and well known) application of change of coordinates.

\begin{proposition}\label{prof: shift_linear_operator}
If $S, X \in \clb(H^2(\D))$, then:

(i) $S$ is a unilateral shift if and only if there exists a unitary $U \in \clb(H^2(\D))$ such that $S = U^* M_z U$.

(ii) If $S$ is a unilateral shift, then $SX = XM_z$ if and only if there exist $\theta\in H^\infty(\D)$ such that $X = U^*M_\theta$, where $U$ is as in (i).
\end{proposition}
\begin{proof}
To prove (i), assume that $S \in \clb(H^2(\D))$ is a unilateral shift. Then there exists $\{f_n\}_{n\in \Z_+} \in B_{H^2(\D)}$ such that $S f_n= f_{n+1}$. Define the unitary operator $U$ on $H^2(\mathbb{D})$ by
\[
U f_n = z^n \qquad (n \in \Z_+).
\]
Clearly, $S=U^*M_zU$. The converse part is straightforward. Part (ii) follows from (i) that $S = U^* M_z U$, and the fact that $M_z(UX) = (UX) M_z$ if and only if $UX = M_\theta$ for some $\theta\in H^\infty(\D)$.
\end{proof}

Along with the unitary $U$ as above, we now return to the issue of anti-linear operators.

\begin{proposition}
Let $X \in \clb_a(H^2(\mathbb{D}))$, and let $S \in \clb(H^2(\mathbb{D}))$ be a unilateral shift. Then $XM_z=S X$ if and only if there exists $\theta\in H^\infty(\D)$ such that $X= U^* M_\theta J_{H^2(\D)}$.
\end{proposition}
\begin{proof}
In view of Proposition \ref{prof: shift_linear_operator}, we have $S=U^*M_zU$. Observe that $XM_z=S X$ if and only if $XM_z J_{H^2(\D)} =S X J_{H^2(\D)}$, which is equivalent to
\[
(X J_{H^2(\D)}) M_z = S (X J_{H^2(\D)}).
\]
Since $X J_{H^2(\D)} \in \clb(H^2(\D))$, by Proposition \ref{prof: shift_linear_operator} it follows that the above equality is equivalent to $X J_{H^2(\D)} = U^*M_\theta$, that is, $X=U^*M_\theta J_{H^2(\D)}$ for some $\theta \in H^\infty(\D)$. This completes the proof of the lemma.
\end{proof}

Recall the canonical factorization of conjugations (see Proposition \ref{Prop: C = UC =CU*}): If $C \in \clc(\clh)$, then there is a unique unitary $V \in \clb(\clh)$ such that
\[
C = V J_{\clh} = J_{\clh} V^*.
\]
Also recall from Lemma \ref{lemma: C = UC =CU*} that if $f_n:=U e_n$ for all $n \in \Z_+$, then
\[
S:= C S_{\clh} C,
\]
is a unilateral shift corresponding to $\{f_n\} \in B_{\clh}$. Note also that $S= C S_{\clh} C$ is equivalent to $CS = S_{\clh} C$. This clearly motivates the question of the classification of conjugations that intertwine $M_z$ and shifts of multiplicity one. The following is our answer to this question in a general setting:

\begin{theorem}\label{Main:theorem}
Let $C \in \clc(\clh)$, and let $S \in \clb(\clh)$ be the unilateral shift corresponding to $\{f_n\}_{n\in \Z_+} \in B_{\clh}$. Then
\[
C S_{\clh} = SC,
\]
if and only if there exists a constant $\lambda$ of unit modulus such that
\[
C = \lambda U J_{\clh},
\]
where $U$ on $\clh$ is the unitary defined by $U e_n = f_n$, $n \in \Z_+$.
\end{theorem}

\begin{proof}
If $C = \lambda U J_{\clh}$, then $C S_{\clh} = \lambda U S_{\clh} J_{\clh}$, and hence
\[
C S_{\clh} e_n = \lambda U S_{\clh} e_n = \lambda U e_{n+1} = \lambda f_{n+1},
\]
where, on the other hand, $SC = \lambda S U J_{\clh}$, and hence
\[
SC e_n = \lambda SU e_n = \lambda S f_n = \lambda f_{n+1},
\]
for all $n \in \Z_+$. This proves that $C S_{\clh} = SC$. For the reverse direction, suppose $C S_{\clh} = SC$, that is, $S_{\clh} = C S C$. Consider the canonical factorization of $C$ as
\[
C = V J_{\clh} = J_{\clh} V^*,
\]
where $V$ is a unique unitary on $\clh$. Then, as in the proof of Lemma \ref{lemma: C = UC =CU*}, we have
\[
S_{\clh} = V^* S V.
\]
For each $n \in \Z_+$, we compute
\[
\begin{split}
S_{\clh} (V^* U e_n) & = V^* S V(V^* U e_n)
\\
& = V^*S U e_n
\\
& = V^*S f_n
\\
& = V^* f_{n+1}
\\
& = (V^* U ) U^* f_{n+1}
\\
& = (V^* U e_{n+1}).
\end{split}
\]
Therefore, $\{V^* U e_n\}_{n\in \Z_+} \in B_{\clh}$ and $S_{\clh} (V^* U e_n) = (V^* U e_{n+1})$ for all $n \in \Z_+$. Similarly, since $S_{\clh} e_n = e_{n+1}$, $n \in \Z_+$, by the definition of unilateral shifts, we have
\[
\ker S_{\clh}^* = \C e_0 = \C (V^* U e_0).
\]
Then there exists a constant $\lambda$ of unit modulus such that
\[
e_0 = \lambda (V^* U e_0).
\]
Also
\[
\begin{split}
e_1 & = S_{\clh} e_0
\\
& = \lambda S_{\clh}(V^* U e_0)
\\
& = \lambda (V^* U e_1).
\end{split}
\]
Then, by induction, we conclude that
\[
e_n = \lambda (V^* U e_{n}) \qquad (n \in \Z_+).
\]
Equivalently, we have
\[
V = \lambda U,
\]
and hence $C = \lambda U J_{\clh}$. This completes the proof of the theorem.
\end{proof}

In particular, if $S \in \clb(H^2(\D))$ is the unilateral shift corresponding to $\{f_n\}_{n\in \Z_+} \in B_{H^2(\D)}$, then a conjugation $C \in \clc(H^2(\D))$ satisfies $C M_z = SC$ if and only if
\[
C=\lambda U^*J_{H^2(\D)},
\]
where $\lambda$ is a constant of unit modulus and $U$ is the unitary defined by $U f_n = z^n$, $n \in \Z_+$.

\vspace{0.3in}

\noindent\textit{Acknowledgement:} The second author is supported in part by the J. C. Bose fellowship. The third author is supported in part by the Core Research Grant (CRG/2019/000908), SERB, Department of Science \& Technology (DST), Government of India.

\end{document}